\documentclass[10pt]{amsart}

\title[Congruences on KGW]{Congruences on K-theoretic \\ Gromov--Witten invariants}

\author[Gu\'er\'e]{J\'er\'emy Gu\'er\'e}
\address{Univ. Grenoble Alpes, CNRS, IF, 38000 Grenoble, France}
\email{jeremy.guere@univ-grenoble-alpes.fr}

\usepackage{amsrefs}
\usepackage{amsfonts}
\usepackage{amsmath}
\usepackage{amssymb}
\usepackage{amscd}
\usepackage{array}
\usepackage{subfigure}
\usepackage{tikz}
\usetikzlibrary{arrows}

\usepackage[all]{xy}

\allowdisplaybreaks[3]
\newcommand{\EE}{\mathbb{E}}
\newcommand{\Contr}{\mathrm{Contr}}
\newcommand{\fix}{\mathrm{fix}}
\newcommand{\cSy}{\mathrm{Sym}}
\newcommand{\loc}{\mathrm{loc}}

\newcommand{\cC}{\mathcal{C}}
\newcommand{\cL}{\mathcal{L}}
\newcommand{\cN}{\mathcal{N}}

\newcommand{\vir}{\mathrm{vir}}

\newcommand{\fc}{\mathfrak{c}}
\newcommand{\PV}{\mathbf{PV}}

\newcommand{\cE}{\mathcal{E}}

\newcommand{\aA}{\mathbb{A}}

\newcommand{\ci}{\mathrm{i}}
\newcommand{\PP}{\mathbb{P}}
\newcommand{\CC}{\mathbb{C}}
\newcommand{\ZZ}{\mathbb{Z}}
\newcommand{\NN}{\mathbb{N}}

\newcommand{\QQ}{\mathbb{Q}}

\newcommand{\cO}{\mathcal{O}}
\newcommand{\ev}{\mathrm{ev}}

\renewcommand{\(}{\left(}
\renewcommand{\)}{\right)}

\newcommand{\cF}{\mathcal{F}}

\newcommand{\cM}{\mathcal{M}}

\newcommand{\cP}{\mathcal{P}}

\newcommand{\sS}{\mathcal{S}}

\newcommand{\cX}{\mathcal{X}}

\newcommand{\Ch}{\mathrm{Ch}}
\newcommand{\Td}{\mathrm{Td}}

\newcommand{\cvir}{c_{\mathrm{vir}}}

\DeclareMathOperator{\Aut}{Aut}

\theoremstyle{plain}

\newtheorem{thm}{Theorem}[section]
\newtheorem{pro}[thm]{Proposition}

\newtheorem*{lem*}{Lemma}

\newtheorem{cor}[thm]{Corollary}

\theoremstyle{definition}

\newtheorem{dfn}[thm]{Definition}
\newtheorem{nt}[thm]{Notation}

\newtheorem{rem}[thm]{Remark}
\newtheorem*{rem*}{Remark}
\newtheorem*{rems*}{Remarks}
\newtheorem{exa}[thm]{Example}
\newtheorem*{exa*}{Examples}
\newtheorem*{exait*}{\rm \em Example}
\newtheorem*{exadefit*}{\rm \em Example/Definition}
\newtheorem*{cla*}{\rm \em Claim}

\newtheorem*{dfn*}{Definition}

\usepackage{amsxtra}

\def\<{\left\langle}
\def\>{\right\rangle}

\begin{document}
\begin{abstract}
We study K-theoretic Gromov--Witten invariants of projective hypersurfaces using a virtual localization formula under finite group actions.
In particular, it provides all K-theoretic Gromov--Witten invariants of the quintic threefold modulo $41$, up to genus $19$ and degree $40$.
As an illustration, we give an instance in genus one and degree one. 
Applying the same idea to a K-theoretic version of FJRW theory, we determine it modulo $41$ for the quintic polynomial with minimal group and narrow insertions, in every genus.
\end{abstract}

\maketitle

\tableofcontents

\setcounter{section}{-1}

\section{Introduction}
One of the first achievement of Gromov--Witten (GW) theory is the celebrated formula of Candelas--de la Ossa--Green--Parkes \cite{Candelas} computing genus-$0$ invariants of the quintic threefold in terms of an hypergeometric series solution of a Picard--Fuchs equation. It was a first instance of mirror symmetry and was proved by \cite{Gi,LLY}.

The K-theoretic version of GW theory, that we refer to as KGW theory, was constructed in \cite{YP}, and it is only recently that mirror symmetry in this context was developed by Givental in his series of papers \cite{Givental3}. It relates the KGW generating series to a $q$-hypergeometric function solution of a finite-difference equation.

Both GW and KGW theories rely on the notion of a perfect obstruction theory \cite{BF}, producing two fundamental objects on the moduli space $\cM_{g,n}(X,\beta)$ of stable maps to a given variety $X$, namely the virtual cycle $[\cM_{g,n}(X,\beta)]^\vir$ living in the Chow ring of the moduli space and the virtual structure sheaf $\cO^\vir_{\cM_{g,n}(X,\beta)}$ living in its K-theory.

Given insertions $y_i \in \mathrm{CH}_*(X)$, or $Y_i \in K^0(X)$, and psi-classes $\psi_i \in \mathrm{CH}^1(\overline{\cM}_{g,n})$, or psi-bundles $\Psi_i \in K^0(\overline{\cM}_{g,n})$, we then form
$$a:= \prod_{i=1}^n \ev^*(y_i) \cdot \psi_i^{d_i} \quad \textrm{and} \quad A := \bigotimes_{i=1}^n \ev^*(Y_i) \otimes \Psi_i^{\otimes d_i} ~,~~ d_i \in \ZZ,$$
using evaluation maps.
It yields GW and KGW invariants
$$p_*\(a \cdot \left[\cM_{g,n}(X,\beta)\right]^\vir\) \in \QQ \quad \textrm{and} \quad p_!\(A \otimes \cO^\vir_{\cM_{g,n}(X,\beta)}\) \in \ZZ,$$
where $p$ is the projection map to a point along which we take pushforwards in Chow or in K-theory\footnote{In K-theory, it is also known as the Euler characteristic.}.
Each theory has an important feature: the virtual cycle is pure-dimensional, leading to a degree condition on the insertions for the GW invariant to be non-zero, and KGW invariants are all integers.
Moreover, the two theories are related via an Hirzebruch--Riemann--Roch theorem, see \cite{Tonita} and \cite[Part IX]{Givental3}, saying that all KGW invariants can be reconstructed from GW invariants.

Let $T$ be a torus.
When the variety $X$ carries a non-trivial $T$-action, so does the moduli space of stable maps, and the virtual cycle and virtual structure sheaf are $T$-equivariant.
One then benefits from the virtual localization formula \cite{GraberPandha} to reduce the computation of invariants to the $T$-fixed locus, which greatly simplifies the calculation.
Unfortunately, the automorphism group of a smooth projective hypersurface such as the quintic threefold is finite (except in the special cases of quadrics, elliptic curves, and K3 surfaces), so that there is no non-trivial $T$-action.

Let $G$ be a finite cyclic group.
In this paper, we take advantage of the fact that the (virtual) localization formula holds with no change under finite group actions.
Since projective hypersurfaces $X$ admit such actions, we can apply it to the study of KGW theory of $X$.
However, we still have two difficulties.
First, the $G$-fixed moduli space is in general quite involved and we cannot guarantee that it is smooth, so even after applying the virtual localization formula, we may not be able to finish the computation.
Second, the (virtual) localization formula gives an answer in a localized ring.
For instance, the $G$-equivariant K-theory of a point is isomorphic to the representation ring $R(G)$, which in the case of a finite cyclic group of order $M$ yields $\ZZ[X]/(1-X^M)$.
Instead of providing an answer in $R(G)$, the (virtual) localization formula only gives us the image in the localized complexified ring $R(G)_{\CC,\loc}$, where we invert a maximal ideal corresponding to a non-zero element in $G$.
The issue with the localized ring is that the map $R(G) \to R(G)_{\CC,\loc}$ is in general not injective.
In our example, we have $R(G)_{\CC,\loc} \simeq \CC$ and the map sends $X$ to a given primitive $M$-th root of unity, so that the non-zero polynomial $1+X+\dotsb+X^{M-1} \in R(G)$ is sent to $0 \in \CC$.
Notice that when the group is a torus $T$, the localization map $R(T) \to R(T)_{\CC,\loc}$ is injective, that is why we have no such issue in the previous paragraph.

We overcome the first difficulty by means of an `equivariant quantum Lefschetz theorem' that we developed for GW theory in \cite[Section 2]{Guere10} and that we adapt to KGW theory and to finite group actions in Section \ref{QLsection}, see Theorem \ref{quantum Lefschetz}.
It compares the $G$-equivariant virtual structure sheaf of an hypersurface $X \subset \PP^N$ to the one of the ambient space $\PP^N$, and then we use the $T$-action on the ambient space to apply the virtual localization formula.
However, Theorem \ref{quantum Lefschetz} requires that for every $G$-fixed stable maps from a curve $C$ to $X$, all stable components of $C$ are contracted to a point in $X$.
This condition could fail if the automorphism group of the curve is too big, leading us to impose restrictions on the genus of the curve and on the degree of the stable map.

The second difficulty is more serious.
Indeed, we know the $G$-equivariant KGW invariant is of the form $a_0+a_1X+\dotsb+a_{M-1}X^{M-1}$, for some integers $a_i$, and our goal would be the `non-equivariant' limit $a_0+\dotsb+a_{M-1}$, but we only have access to the complex number $a_0+a_1\zeta+\dotsb+a_{M-1}\zeta^{M-1}$, where $\zeta$ is a primitive $M$-th root of unity.
Luckily, KGW are integers, so that when $M$ is a prime number, we can sum all these complex numbers for primitive roots and obtain the KGW invariant modulo $M$.

As a conclusion, we seek for automorphisms of $X$ of prime order with isolated fixed points.
For instance, the quintic threefold can be realized as the zero locus in $\PP^4$ of the loop polynomial $x_0^4x_1+\dotsb+x_4^4x_0$ and the action
$$\zeta \cdot (x_0,\dotsc,x_4) = (\zeta x_0,\zeta^{-4} x_1,\zeta^{16}x_2,\zeta^{-64}x_3,\zeta^{256}x_4) ~,~~ \zeta := e^{\frac{2 \ci \pi}{41}}$$
yields an automorphism of $X$ of prime order $41$, whose fixed points are coordinate points.
As a result of Corollary \ref{non equiv limit}, we obtain all KGW invariants of the quintic threefold up to genus $19$ and degree $40$, modulo $41$.

\begin{rem}
It happens that $41$ is the biggest prime number $p$ for which there exists an automorphism of order $p$ for a smooth quintic hypersurface, see \cite{Oguiso}.
\end{rem}

In Proposition \ref{exa}, we provide an instance of this calculation in genus one. Precisely, we compute
$$\chi \left(\cfrac{\cO^{\vir}_{\cM_{1,0}(X,1)}}{1-q \EE^\vee}\right) \equiv (120 q^2+180 q +125) \cfrac{(1-q^4-q^6)}{(1-q^4)(1-q^6)} \in \ZZ/205\ZZ[\![q]\!],$$
where $\EE$ denotes the Hodge bundle.

Interestingly, if we can find automorphisms of prime orders for infinitely many primes and if we can handle the respective localization formulas, then we are able to determine KGW invariants as integers instead of modulo a prime number.
We apply this idea to elliptic curves.
There are indeed $p$-torsion points for every prime number $p$, so that a translation by this point is an automorphism of order $p$.
Furthermore, the localization formula is trivial since there are no fixed points.
We deduce the vanishing of all KGW invariants of an elliptic curve, with homogeneous insertions.

Similarly to GW theory, Fan--Jarvis--Ruan developped a quantum singularity theory for Landau--Ginzburg orbifolds, see \cite{FJRW1,FJRW2}. It is known as the FJRW theory and an algebraic construction has been established by Polishchuk--Vaintrob in \cite{Polish1}.
Precisely, they construct a matrix factorization over the moduli space of $(W,G)$-spin curves, where $W$ is a non-degenerate quasi-homogeneous polynomial and $G$ is an admissible group of symmetries. We refer to \cite{Guere1} for details.

In Section \ref{FJRWth}, we explain how to construct a K-theoretic version of FJRW theory and we then pursue the same goal as for KGW theory: compute invariants by applying the localization theorem under finite group actions.
We focus on the quintic polynomial with group $\mu_5$ for clarity of the exposition and we find all its K-theoretic FJRW invariants with narrow insertions in every genus and modulo $41$, see Corollary \ref{nonequiv formul}.
Here, we do not have restriction bounds on the genus of the curve.

In \cite{Guere1}, we compute genus-$0$ FJRW invariants of chain polynomials using a characteristic class $\fc_t \colon K^0(S) \to H^*(S)[\![t]\!]$, that we can define for a line bundle $L$ over a smooth DM stack $S$ as
$$\fc_t(L) = \Ch\(\lambda_{-1} L^\vee\) \Td(L) = c_1(L) \cdot \cfrac{e^{c_1(L)}-t}{e^{c_1(L)}-1},$$
and then extend it multiplicatively.
Genus-$0$ FJRW invariants of a chain polynomial $x_1^{a_1}x_2+\dotsb+x_N^{a_N}$ are then expressed as
\begin{equation}\label{lim}
c_\vir = \lim_{t \to 1} \prod_{j=1}^N \fc_{t_j}(-R\pi_*\cL_j) ~,~~ g=0,
\end{equation}
where $t_j:=t^{(-a_1) \dotsm (-a_{j-1})}$ and $R\pi_*\cL_j$ are the derived pushforwards of the  universal line bundles over the moduli space of $(W,G)$-spin curves.
It is remarkable that such a limit exists and the author wondered since then whether other limits could exist, for instance when $t$ tends to some root of unity.
Interestingly, we prove for the quintic loop polynomial with group $\mu_5$ and narrow insertions that such a limit exists for all genus when $t$ tends to a $41$-th root of unity $\zeta_{41}$.
It then converges to a $\ZZ/41\ZZ$-equivariant version of the FJRW virtual cycle, defined as follows.
The two-periodic complex obtained from the Polishchuk--Vaintrob matrix factorization naturally decouples as a direct sum of $41$ two-periodic complexes\footnote{Precisely, the $k$-th two-periodic complex is the one containing vector bundles $\cSy^{k+41 \cdot l} A_1^\vee$, with $l \in \NN$, in the notations of \cite{Guere1}.}.
Each one of them provides a (virtual) cycle $a_k$, $0 \leq k \leq 40$, and we define
$$\cvir^{\ZZ/41\ZZ} := \sum_{k=0}^{40} a_k \zeta_{41}^k = \lim_{t \to \zeta_{41}} \prod_{j=1}^N \fc_{t_j}(-R\pi_*\cL_j) ~,~~ g \geq 0.$$
We easily find similar results for other loop polynomials.

As a conclusion, we mention a future line of research.
In \cite{LG/CYquintique,LGCY}, the authors study the so-called genus-$0$ Landau--Ginzburg/Calabi--Yau (LG/CY) correspondence, which provides a striking relation between GW theory of a projective hypersurface and FJRW theory of the defining polynomial.
Following \cite{Chiodo2}, there is a similar correspondence in higher genus as well.
Since we expect the LG/CY correspondence to hold in K-theory as well, it would be interesting to probe a K-theoretic version for the quintic threefold, up to genus $19$, degree $40$, and modulo $41$.

Another question we may ask is: what information do we get on GW invariants of the quintic threefold up to genus $19$?
The quintic threefold $X$ is special, its virtual cycle (with no markings) is $0$-dimensional, so that a lot of its GW invariants vanishes.
In fact, they are all deduced from some rational numbers $n_{g,d} \in \QQ$, for non-negative integers $g$ and $d$, corresponding to its GW invariants without markings.
As a consequence, we expect some simplifications in the Hirzebruch--Riemann--Roch theorem \cite{Tonita,Givental3} and to find formulas expressing KGW invariants of $X$ in terms of $n_{g,d}$'s.
Moreover, it is proven in \cite{Honglu5,Ruan5,Li5} that all values of $n_{g,d}$ are expressed in terms of low degrees, where $5d \leq 2g-2$.
Up to genus $19$, there are exactly $61$ unknowns:
$$n_{4,1}, n_{5,1},\dotsc, n_{18,6}, n_{19,7} \in \QQ.$$
As we are able to compute all KGW invariants modulo $41$ up to genus $19$ and degree $40$, then we expect a lot of relations among these $61$ unknows.
Moreover, KGW is not restricted by a degree condition on insertions, so we can also insert K-classes from $\PP^4$, yielding indeed infinitely many relations among these $61$ unknowns. Of course, we do not know yet how many of these relations are non-trivial.
It would also be enlightening to express KGW invariants in terms of BPS numbers, which are integers as well, see \cite{phys} for a formula in genus zero.

%
%
%

~~

\noindent
\textbf{Notations.}
In this paper, we work over the complex numbers.
We denote by $G_0(X)$ the Grothendieck group of coherent sheaves on a DM stack $X$ and by $K^0(X)$ the Grothendieck ring of vector bundles on $X$.
If a linear algebraic group $G$ acts on $X$, then we denote by $G_0(G,X)$ and $K^0(G,X)$ the Grothendieck groups of $G$-equivariant coherent sheaves and vector bundles.
They are identified when $X$ is smooth, by \cite{Tho1}.
When $X$ is a point, then it equals the representation ring $R(G)$ of the group $G$.
The $G$-fixed locus inside $X$ is denoted by $X^G$.
For an element $h \in G$, we denote by $X^h$ the $h$-fixed locus.
If $V$ is a $G$-equivariant vector bundle over $X$, then we denote by $\lambda_{-1}^G(V) = \sum_{k \geq 0} (-1)^k \Lambda^k V \in K^0(G,X)$ the lambda-structure in K-theory.
When we forget the group action, we simply denote it by $\lambda_{-1}(V) \in K^0(X)$.

Let $G$ be a diagonalizable group. The complexified representation ring $R(G)_\CC := R(G) \otimes \CC$ is identified with the coordinate ring $\cO(G)$ of $G$. Hence, for every $h \in G$, there is a corresponding maximal ideal $\mathfrak{m}_h \subset R(G)_\CC$.
Denote by $G_0(G,X)_\loc := G_0(G,X)_{\mathfrak{m}_h}$ and $K^0(G,X)_\loc := K^0(G,X)_{\mathfrak{m}_h}$ the localizations.
Assume $X$ is smooth and let $\iota \colon X^h \subset X$ be the inclusion of the $h$-fixed locus.
The localization theorem says
\begin{equation}\label{loca form}
A = \iota_! \cfrac{\iota^*A}{\lambda_{-1}^G(N_\iota^\vee)} \in K^0(G,X)_\loc ~,~~ \forall A \in K^0(G,X)_\loc,
\end{equation}
see \cite{Tho2}.
It is in particular true for finite groups $G$, even though the localization map $R(G)_\CC \to R(G)_{\loc}$ is not injective in that case.
Moreover, we can relax the smoothness condition on $X$.
Indeed, if $X$ is singular but carries a $G$-equivariant perfect obstruction theory $[E_{-1} \to E_0]$, then there is a $G$-equivariant virtual structure sheaf $\cO_X^{\vir,G} \in G_0(G,X)$, see \cite{YP}.
The obstruction theory pulls-back to the $G$-fixed locus $X^G$ and we denote by $N^\vir_\iota \in K^0(G,X^G)$ the K-theoretic class of the dual of its $G$-moving part.
The $G$-fixed part gives a perfect obstruction theory on $X^G$ and yields a virtual structure sheaf $\cO^\vir_{X^G}$.
Furthermore, we have the virtual localization formula
\begin{equation}\label{vir loca form}
\cO_X^{\vir,G} = \iota_! \( \cfrac{\cO_{X^G}^\vir}{\lambda_{-1}^G({N^\vir_\iota}^\vee)} \) \in G_0(G,X)_\loc.
\end{equation}
In particular, it applies to the moduli space $\cM(\cX)$ of stable maps to a smooth DM stack $\cX$. 
Here, we specify the genus, the degree, and the number of markings as $\cM_{g,n}(\cX,\beta)$ when needed.

The letters GW stand for Gromov--Witten and KGW for K-theoretic Gromov--Witten.

\section{Equivariant quantum Lefschetz theorem}\label{QLsection}
This section is a generalization of \cite[Section 2]{Guere10} to $K$-theory and to more general group actions.
The main result is an `equivariant quantum Lefschetz' theorem which is of first importance in the next section.

\subsection{Virtual localization formula}
Let $G$ be a linear algebraic group and $\cX$ be a smooth DM stack equipped with a $G$-action.
The moduli space $\cM(\cX)$ of stable maps to $\cX$ carries a $G$-action, a $G$-equivariant perfect obstruction theory, and thus a $G$-equivariant virtual structure sheaf $\cO^{\vir,G}_{\cM(\cX)} \in G_0(G,\cM(\cX))$.

Denote by $\iota \colon \cM(\cX)^G \hookrightarrow \cM(\cX)$ the embedding of the $G$-fixed locus.
By definition, the virtual normal bundle $N_\iota^\vir \in K^0(G,\cM(X)^G)$ is the moving part of the pull-back of the perfect obstruction theory to the fixed locus
\footnote{The perfect obstruction theory can be realized as a complex of vector bundles, hence the virtual normal bundle is an element in $K^0$ rather than an element in $G_0$.}.
The virtual localization formula \eqref{vir loca form} states
$$\cO^{\vir,G}_{\cM(\cX)} = \iota_! \( \cfrac{\cO^\vir_{\cM(\cX)^G}}{\lambda^G_{-1}({N_\iota^\vir}^\vee)} \) \in G_0(G,\cM(\cX))_\loc.$$

\subsection{Enhancement of the group}\label{sec enh}
Let $G \subset T$ be an embedding of linear algebraic groups and $\cX \hookrightarrow \cP$ be an embedding of smooth DM stacks equipped with a $G$-action.
We assume that
\begin{itemize}
\item the $G$-fixed loci of $\cX$ and of $\cP$ are equal,
\item for every $G$-fixed stable map to $\cP$, all stable components of the source curve are sent to $\cP^G$,
\item $\cP$ is equipped with a $T$-action extending the $G$-action,
\item the normal bundle of $\cX \hookrightarrow \cP$ is the pull-back of a $T$-equivariant vector bundle $\cN$ over $\cP$,
\item the vector bundle $\cN$ is convex up to two markings, i.e.~for every stable map $f \colon \cC \to \cP$ where $\cC$ is a smooth genus-$0$ orbifold curve with at most two markings we have $H^1(\cC,f^*\cN)=0$.
\end{itemize}

Let us first consider the $G$-fixed loci of the moduli spaces of stable maps and observe the following fibered diagram

\begin{center}
\begin{tikzpicture}[scale=0.75]	
\node (A') at (0,1.5) {$\cM(\cX)^{G}$};
\node (B') at (3,1.5) {$\cM(\cP)^{G}$};
\node (C') at (3,0) {$\cM(\cP)$};
\node (D') at (0,0) {$\cM(\cX)$};
\draw[->] (D') -- (C');
\draw[->] (A') -- (B');
\draw[->] (A') -- (D');
\draw[->] (B') -- (C');
\draw[-] (1.5-0.1,0.75-0.1) -- (1.5-0.1,0.75+0.1) -- (1.5+0.1,0.75+0.1) -- (1.5+0.1,0.75-0.1) -- (1.5-0.1,0.75-0.1);
\node[above] at (1.5,1.5){$j$};
\node[right] at (3,0.75){$\widetilde{\iota}$};
\node[left] at (0,0.75){$\iota$};
\node[below] at (1.5,0){$\widetilde{j}$};
\end{tikzpicture}
\end{center}
Writing $i \colon \cX \hookrightarrow \cP$, we have a $G$-equivariant short exact sequence
$$0 \to T_{\cX} \to i^*T_{\cP} \to i^*\cN \to 0$$
inducing a distinguished triangle for the dual of the perfect obstruction theories of $\cM(\cX)$ and of $\cM(\cP)$
\begin{equation}\label{exactseqobstr}
R\pi_*f^*T_{\cX} \to R\pi_*f^*T_{\cP} \to R\pi_*f^*\cN \to \(R\pi_*f^*T_{\cX}\)[1].
\end{equation}

The term $\cE := R\pi_*f^*\cN$, pulled-back to $\cM(\cP)^G$, has a fixed and a moving part, that we denote respectively by $\cE_\mathrm{fix}$ and $\cE_\mathrm{mov}$.

\begin{pro}
The fixed part $\cE_\mathrm{fix}$ is a vector bundle over the fixed moduli space $\cM(\cP)^G$.
\end{pro}

\begin{proof}
Let $f \colon \cC \to \cP$ be a stable map belonging to $\cM(\cP)^G$.
We denote by $\rho \colon \cC \to C$ the coarse map.
It is enough to prove
$$H^1(C,\rho_*f^* \cN)^\mathrm{fix}=0.$$

Take the normalization $\nu \colon C^\nu \to C$ of the curves at all their nodes. We have
$$C^\nu = \bigsqcup_{i \in I} C^\mathrm{fix}_i \sqcup \bigsqcup_{j \in J} C^\mathrm{nf}_j,$$
where the upper-scripts refer respectively to fixed/non-fixed components of $C^\nu$ under the map $f$.
In particular, non-fixed components are unstable curves, i.e.~the projective line with one or two special points.
By the normalization exact sequence, we obtain an exact sequence
$$\bigoplus_{\mathrm{nodes}} H^0(\mathrm{node},{f^* \cN}_{|\mathrm{node}}) \to H^1(C,f^* \cN) \to  H^1(C^\nu,\nu^*f^* \cN) \to 0,$$
with
$$H^1(C^\nu,\nu^*f^* \cN) = \bigoplus_{i \in I} H^1(C^\mathrm{fix}_i,\nu^*f^* \cN) \oplus \bigoplus_{j \in J} H^1(C^\mathrm{nf}_j,\nu^*f^* \cN).$$
Since the normal bundle has a non-trivial $G$-action once restricted to the fixed locus of $\cX$ (or equivalently of $\cP$),
then we have
$$H^0(\mathrm{node},{f^* \cN}_{|\mathrm{node}})^\mathrm{fix}=0 \quad \mathrm{and} \quad H^1(C^\mathrm{fix}_i,\nu^* f^* \cN)^\mathrm{fix}=0.$$
Therefore, it remains to see the vanishing of $H^1$ for non-fixed unstable curves $C_j^\mathrm{nf}$, $j \in J$.
The curve $C_j^\mathrm{nf}$ is isomorphic to $\PP^1$ with either one or two markings, hence $H^1(C^\mathrm{nf}_j,\nu^*f^* \cN)=0$ by our assumption of convexity up to two markings.
\end{proof}

Denote by $\cO_{\cM(\cP)^G}^\vir$ the virtual structure sheaf obtained by the $G$-fixed part of the perfect obstruction theory $R\pi_*f^*T_{\cP}$.

\begin{pro}
We have
$$j_!\cO_{\cM(\cX)^G}^\vir = \lambda^G_{-1}\(\cE_\mathrm{fix}^\vee\) \otimes \cO_{\cM(\cP)^G}^\vir \in G_0(\cM(\cP)^G).$$
Furthermore, in the localized equivariant K-theoretic ring, we have
$$\cfrac{1}{\lambda^G_{-1}\({N^\vir_\iota}^\vee\)} = j^* \(\cfrac{\lambda^G_{-1}({\cE_\mathrm{mov}}^\vee)}{\lambda^G_{-1}({N^\vir_{\widetilde{\iota}}}^\vee)}\) \in K^0(G,\cM(\cP)^G)_\loc.$$
\end{pro}

\begin{proof}
It follows from the standard proof using convexity and we recall here the main arguments.
	
The variety $\cX$ is the zero locus of a section of the vector bundle $\cN$ over the ambient space $\cP$.
This section induces a map $s$ from the moduli space of stable maps to $\cP$ to the direct image cone $\pi_*f^*\cN$, see \cite[Definition $2.1$]{Li}.
Since the moduli space $\cM(\cP)^G$ is fixed by the action of $G$, then it maps to the fixed part of the direct image cone, that is the vector bundle $\cE_\mathrm{fix}$.
Hence we have the fibered diagram
\begin{center}
	\begin{tikzpicture}[scale=0.75]	
	\node (A') at (0.5+6+0,1.5) {$\cM(\cX)^G$};
	\node (B') at (-0.5+1+6+4.5,1.5) {$\cM(\cP)^G$};
	\node (C') at (-0.5+1+6+4.5,0) {$\cE_\mathrm{fix}$};
	\node (D') at (0.5+6+0,0) {$\cM(\cP)^G$};
	\draw[->] (D') -- (C');
	\draw[->] (A') -- (B');
	\draw[->] (A') -- (D');
	\draw[->] (B') -- (C');
	\draw[-] (1.5+6+1.25-0.1,0.75-0.1) -- (1.5+6+1.25-0.1,0.75+0.1) -- (1.5+6+1.25+0.1,0.75+0.1) -- (1.5+6+1.25+0.1,0.75-0.1) -- (1.5+6+1.25-0.1,0.75-0.1);
	\node[above] at (6.25+1.25+1.3,1.5){$j$};
	\node[right] at (6.5+2.5+2,0.75){$s$};
	\node[below] at (6.25+1.25+1.3,0){$0$};
	\end{tikzpicture}
\end{center}
where the bottom map is the embedding as the zero section.
The fixed part of the distinguished triangle \eqref{exactseqobstr} gives a compatibility datum of perfect obstruction theories for the fixed moduli spaces.
Functoriality of the virtual structure sheaf gives
$$j^! \cO_{\cM(\cP)^G}^\vir = \cO_{\cM(\cX)^G}^\vir.$$
Applying the projection formula via the map $j$ on both sides and using the Koszul resolution gives the first result.
The second part of the statement follows from the moving part of the distinguished triangle \eqref{exactseqobstr}.
\end{proof}

Eventually, by the virtual localization formula, the $G$-equivariant virtual structure sheaf satisfies
\begin{eqnarray*}
\widetilde{j}_! \cO_{\cM(\cX)}^{\vir,G} & = & \widetilde{j}_!\iota_! \left( \cfrac{\cO_{\cM(\cX)^G}^\vir}{\lambda^G_{-1} ({N^\vir_\iota}^\vee)} \right) \\
& = & \widetilde{\iota}_!j_! \left(\cO_{\cM(\cX)^G}^\vir \otimes j^* \(\cfrac{\lambda^G_{-1}({\cE_\mathrm{mov}}^\vee)}{\lambda^G_{-1}({N^\vir_{\widetilde{\iota}}}^\vee)}\) \right) \\
& = & \widetilde{\iota}_! \left(\lambda^G_{-1}\({\cE_\mathrm{fix}}^\vee\) \otimes \cO_{\cM(\cP)^G}^\vir \otimes  \cfrac{\lambda^G_{-1}({\cE_\mathrm{mov}}^\vee)}{\lambda^G_{-1}({N^\vir_{\widetilde{\iota}}}^\vee)} \right) \\
& = & \widetilde{\iota}_! \left(\lambda^G_{-1}\(\cE^\vee\) \otimes \cfrac{\cO_{\cM(\cP)^G}^\vir}{\lambda^G_{-1}({N^\vir_{\widetilde{\iota}}}^\vee)} \right), \\
\end{eqnarray*}
where equalities happen in $G_0(G,\cM(\cP))_\loc$.
\begin{rem}\label{defafterloc}
If it were defined, the right-hand side would equal
$$\lambda^G_{-1}\(R\pi_*f^*\cN\)^\vee \otimes \cO_{\cM(\cP)}^{\vir,G},$$
using the virtual localization formula, but it is not clear that the $G$-lambda class of $R\pi_*f^*\cN$ is defined in $G_0(G,\cM(\cP))_\loc$.
However, we say that $\lambda^G_{-1}\(R\pi_*f^*\cN\)^\vee$ is defined after localization\footnote{We find this definition for the formal quintic, see \cite{Lho}.} to mean that its pull-back to the fixed locus is defined.
\end{rem}

Now, we aim to extend the right-hand side of the equality to the $T$ action.
The inclusion of groups $G \hookrightarrow T$ yields a morphism
$$\xi_* \colon G_0(T,\cM(\cP)) \to G_0(G,\cM(\cP)),$$
under which we get
$$\xi_* \left( \cO_{\cM(\cP)}^{\vir,T} \right) = \cO_{\cM(\cP)}^{\vir,G}.$$
Unfortunately, the map $\xi_*$ is only partially defined when we localize equivariant parameters: 
the denominators could be non-zero in the $T$-localization but vanish in the $G$-localization.
It is easier to work out this issue on the fixed locus of the moduli space.

Let $\cM(\cP)^T \hookrightarrow \cM(\cP)$ denote the $T$-fixed locus of the moduli space. In particular, we have the inclusion $\widehat{\iota} \colon \cM(\cP)^T \hookrightarrow \cM(\cP)^G$.
We notice that the moduli space $\cM(\cP)^G$ is stable under the $T$-action from $\cM(\cP)$ and that the map $\widehat{\iota}$ is $T$-equivariant.
Moreover, we have a $T$-equivariant virtual structure sheaf
$$\cO_{\cM(\cP)^G}^{\vir,T} \in G_0(T,\cM(\cP)^G)$$
and the equality $$\xi_*\(\cO_{\cM(\cP)^G}^{\vir,T}\) = \cO_{\cM(\cP)^G}^{\vir} \in G_0(G,\cM(\cP)^G).$$

By the virtual localization formula, we have
$$\cO_{\cM(\cP)^G}^{\vir,T} = \widehat{\iota}_! \(\cfrac{\cO_{\cM(\cP)^T}^\vir}{\lambda^T_{-1}({N^\vir_{\widehat{\iota}}}^\vee)} \) \in G_0(T,\cM(\cP)^G)_\loc.$$
Furthermore, we have the following equality in K-theory on the space $\cM(\cP)^T$
\begin{equation}\label{virtnormbund}
N^\vir_{\widetilde{\iota} \circ \widehat{\iota}} = \widehat{\iota}^*N^\vir_{\widetilde{\iota}} + N^\vir_{\widehat{\iota}}.
\end{equation}
Indeed, let $\cF$ be the pull-back of $R\pi_*f^*T\cP$ to $\cM(\cP)^T$. By definition, the virtual normal bundle $N^\vir_{\widetilde{\iota} \circ \widehat{\iota}}$ is the $T$-moving part $\cF^\mathrm{mov}$, which decomposes as $\cF^{\mathrm{mov}} = \cF^{\mathrm{mov}}_\mathrm{fix} + \cF^{\mathrm{mov}}_\mathrm{mov}$, where the subscript denotes the $G$-fixed/moving part.
By definition, the virtual normal bundle $\widehat{\iota}^*N^\vir_{\widetilde{\iota}}$ is the $G$-moving part of $\cF$, i.e.~$\cF^{\mathrm{mov}}_\mathrm{mov}$ since there is no $G$-moving $T$-fixed part in $\cF$.
Eventually, the virtual normal bundle $N^\vir_{\widehat{\iota}}$ identifies with $\cF^{\mathrm{mov}}_\mathrm{fix}$.

\begin{rem}
The virtual normal bundle $N^\vir_{\widetilde{\iota}}$ is defined on $\cM(\cP)^G$ and we have a well-defined equality
$$\xi_*\(\lambda^T_{-1}\({N^\vir_{\widetilde{\iota}}}^\vee\)^{-1}\) = \lambda^G_{-1}\({N^\vir_{\widetilde{\iota}}}^\vee\)^{-1} \in K^0(G,\cM(\cP)^G)_\loc.$$
We also have seen the $G$-decomposition $\cE = \cE_\mathrm{fix} + \cE_\mathrm{mov}$ over $\cM(\cP)^G$ with $\cE_\mathrm{fix}$ being a $T$-equivariant vector bundle.
Indeed, the vector bundle $\cN$ over $\cP$ is $T$-equivariant, thus so are $\cE$ and $\cE_\mathrm{fix}$.
As a consequence, the following equality is well-defined
$$\xi_*\(\lambda^T_{-1}(\cE)\) = \lambda^G_{-1}(\cE) \in K^0(G,\cM(\cP)^G)_\loc.$$
\end{rem}

\begin{pro}\label{proploc}
Consider the well-defined class
$$C_T :=\widehat{\iota}^*\(\lambda^T_{-1}\(\cE^\vee\)\) \otimes \cfrac{\cO_{\cM(\cP)^T}^\vir}{\lambda^T_{-1}({N^\vir_{\widetilde{\iota} \circ \widehat{\iota}}}^\vee)} \in G_0(T,\cM(\cP)^T)_\loc.$$
Then its push-forward under the inclusion $\widehat{\iota}$ equals
$$\widehat{\iota}_!\(C_T\) = \lambda_{-1}^T\(\cE^\vee\) \otimes \cfrac{\cO_{\cM(\cP)^G}^{\vir,T}}{\lambda_{-1}^T({N^\vir_{\widetilde{\iota}}}^\vee)} \in G_0(T,\cM(\cP)^G)_\loc.$$
In particular, we have
$$\xi_*\(\widehat{\iota}_!\(C_T\)\) = \lambda_{-1}^G\(\cE^\vee\) \otimes \cfrac{\cO_{\cM(\cP)^G}^\vir}{\lambda_{-1}^G({N^\vir_{\widetilde{\iota}}}^\vee)} \in G_0(G,\cM(\cP)^G)_\loc.$$
\end{pro}

\begin{proof}
By the virtual localization above and Equation \eqref{virtnormbund}, we have
\begin{eqnarray*}
\widehat{\iota}_!\(C_T\) & = & \lambda_{-1}^T\(\cE^\vee\) \otimes \widehat{\iota}_! \(\cfrac{\cO_{\cM(\cP)^T}^\vir}{\lambda_{-1}^T({N^\vir_{\widetilde{\iota} \circ \widehat{\iota}}}^\vee)} \) \\
& = & \lambda_{-1}^T\(\cE^\vee\) \otimes \widehat{\iota}_! \(\cfrac{\cO_{\cM(\cP)^T}^\vir}{\widehat{\iota}^*\(\lambda_{-1}^T({N^\vir_{\widetilde{\iota}}}^\vee)\) \otimes \lambda_{-1}^T({N^\vir_{\widehat{\iota}}}^\vee)} \) \\
& = & \cfrac{\lambda_{-1}^T\(\cE^\vee\)}{\lambda_{-1}^T({N^\vir_{\widetilde{\iota}}}^\vee)} \otimes \widehat{\iota}_! \(\cfrac{\cO_{\cM(\cP)^T}^\vir}{\lambda_{-1}^T({N^\vir_{\widehat{\iota}}}^\vee)} \) \\
& = & \cfrac{\lambda_{-1}^T\(\cE^\vee\) \otimes \cO_{\cM(\cP)^G}^{\vir,T}}{\lambda_{-1}^T({N^\vir_{\widetilde{\iota}}}^\vee)}. \\
\end{eqnarray*}
The last sentence follows from the following property of $\xi_*$. Let $Z$ be a DM stack with a $T$-action and take $A \in K^0(T,Z)_\loc$, $B \in G_0(T,Z)_\loc$, $a \in K^0(G,Z)_\loc$, and $b \in G_0(G,Z)_\loc$.
If $\xi_*(A)=a$ and $\xi_*(B)=b$ are well-defined equalities, then $\xi_*(A \otimes B)$ is well-defined and equals the localized class $a \otimes b$.
\end{proof}

Eventually, the push-forward maps $\widetilde{\iota}_!$ and $\xi_!$ commute when the later is well-defined. Precisely, the map $\widetilde{\iota}$ is $T$-equivariant and for any localized class $C \in G_0(T,\cM(\cP)^G)_\loc$ such that $\xi_*(C)$ is well-defined in $G_0(G,\cM(\cP)^G)_\loc$, then the localized class $\widetilde{\iota}_! (C)$ is well-defined under $\xi_*$ and we have
$$\widetilde{\iota}_! \xi_*(C) = \xi_* \widetilde{\iota}_! (C) \in G_0(G,\cM(\cP))_\loc.$$

\subsection{Equivariant quantum Lefschetz formula}
Summarizing our discussion, we obtain the following.

\begin{thm}[Equivariant quantum Lefschetz]\label{quantum Lefschetz}
Let $\cX \hookrightarrow \cP$ be a $G$-equivariant embedding of smooth DM stacks satisfying assumptions listed at the beginning of this section.
Then we have
$$\widetilde{j}_!\cO_{\cM(\cX)}^{\vir,G}=\xi_* \( \lambda_{-1}^T\(R\pi_*f^*\cN\)^\vee \otimes \cO_{\cM(\cP)}^{\vir,T} \) \in G_0(G,\cM(\cP))_\loc,$$
where $\widetilde{j}$ is the embedding of moduli spaces and $\xi_*$ is the specialization of $T$-equivariant parameters into $G$-equivariant parameters.
Here, the $T$-equivariant lambda class $\lambda_{-1}^T (R \pi_*f^*\cN)^\vee$ is defined after localization, see Remark \ref{defafterloc}.
\end{thm}

\begin{proof}
Using previous equalities, we get
\begin{eqnarray*}
\widetilde{j}_!\cO_{\cM(\cX)}^{\vir,G}
& = & \widetilde{\iota}_! \left(\lambda_{-1}^G\(\cE^\vee\) \otimes \cfrac{\cO_{\cM(\cP)^G}^\vir}{\lambda_{-1}^G({N^\vir_{\widetilde{\iota}}}^\vee)} \right), \\
& = & \xi_* \widetilde{\iota}_! \widehat{\iota}_! \(  \widehat{\iota}^*\(\lambda_{-1}^T\(\cE^\vee\)\) \otimes \cfrac{\cO_{\cM(\cP)^T}^\vir}{\lambda_{-1}^T({N^\vir_{\widetilde{\iota} \circ \widehat{\iota}}}^\vee)} \).
\end{eqnarray*}
Following Remark \ref{defafterloc}, the meaning of `defined after localization' is precisely
\begin{eqnarray*}
\xi_* \( \lambda_{-1}^T\(R\pi_*f^*\cN\)^\vee \otimes \cO_{\cM(\cP)}^{\vir,T} \) & = & \xi_* \( \lambda_{-1}^T\(R\pi_*f^*\cN\)^\vee \otimes \widetilde{\iota}_! \widehat{\iota}_! \(\cfrac{\cO_{\cM(\cP)^T}^\vir}{\lambda_{-1}^T({N^\vir_{\widetilde{\iota} \circ \widehat{\iota}}}^\vee)}  \) \)\\
& = & \xi_* \widetilde{\iota}_! \widehat{\iota}_! \(  \widehat{\iota}^*\(\lambda_{-1}^T\(\cE^\vee\)\) \otimes \cfrac{\cO_{\cM(\cP)^T}^\vir}{\lambda_{-1}^T({N^\vir_{\widetilde{\iota} \circ \widehat{\iota}}}^\vee)} \).
\end{eqnarray*}
\end{proof}

\section{K-theoretic Gromov--Witten theory}
\subsection{Automorphisms of loop hypersurfaces}\label{loop}
Let $X$ be a smooth degree-$d$ hypersurface in $\PP^N$.
K-theoretic Gromov--Witten (KGW) theory is invariant under smooth deformations, so that we can choose any degree-$d$ homogeneous polynomial $P$ to define $X$, as long as it satisfies the jacobian criterion for smoothness.
Here, we will focus on loop polynomials, i.e.~we take
$$X := \left\lbrace x_0^{d-1}x_1 + \dotsb + x_N^{d-1}x_0 = 0 \right\rbrace \subset \PP^N.$$


Let $\overline{M}:=\cfrac{|1-(1-d)^{N+1}|}{d}$ and consider on $\PP^N$ the $\ZZ/\overline{M}\ZZ$-action
$$\zeta \cdot (x_0,\dotsc,x_N) = (x_0, \zeta x_1, \zeta^{u_2}x_2,\dotsc,\zeta^{u_N}x_N),$$
where $u_0:=0$ and $u_{j+1} := 1-(d-1) u_j$.
We have $(d-1) u_N \equiv 1$ modulo $\overline{M}$, so that the hypersurface $X$ is $\ZZ/\overline{M}\ZZ$-invariant.
Explicitely, we have
$$u_j = \sum_{l=0}^{j-1} (1-d)^l \quad \textrm{and} \quad \overline{M} = \sum_{l=0}^N (1-d)^l.$$

\begin{nt}
Let $\overline{M}_1:=\overline{M}$ and $\overline{M}_{j+1}:=\cfrac{\overline{M}_j}{u_{j+1} \wedge \overline{M}_j}$ for $1 \leq j <N$.
We denote $M:=\overline{M}_N$ and $G:=\ZZ/M\ZZ$.
\end{nt}


\begin{pro}\label{fixloccond}
The group $G$ acts on $\PP^N$, leaves the hypersurface $X$ invariant, and for all non-zero $g \in G$, the $g$-fixed locus in $\PP^N$ consists of all coordinate points.
Furthermore, assuming the Calabi--Yau condition $N+1=d$ and assuming $d$ is a prime number, then we have $M=\overline{M}$. It holds in particular for the quintic hypersurface in $\PP^4$.
\end{pro}

\begin{proof}
By the construction of $M$, we see that every $u_j$, with $1 \leq j \leq N$ is coprime with $M$.
Therefore, we have, for all $0 \leq i < j \leq N$ and all $0 < k < M$,
$$\zeta^{k u_i} \neq \zeta^{k u_j} ~,~~ \zeta = e^{\frac{2 \ci \pi}{M}},$$
hence the statement about the fixed locus.

For the second statement, let $d=N+1$ be a prime number.
Then we have $\overline{M} \equiv 0$ modulo $d$.
Thus, if we have $k u_j \equiv 0$ modulo $\overline{M}$, then we have $k u_j \equiv 0$ modulo $d$.
But $u_j \equiv j$ modulo $d$, so that $k=0$.
As a consequence, every $u_j$ is coprime with $\overline{M}$, and $\overline{M}=M$.
\end{proof}

\begin{exa}\label{quint exa}
We realize the quintic hypersurface in $\PP^4$ as
$$P=x_0^4x_2+\dotsb+x_4^4x_0.$$
Then the group is $G = (\ZZ/205\ZZ)$, acting as
$$\zeta \cdot \underline{x} = (x_0,\zeta x_1,\zeta^{-3}x_2,\zeta^{13}x_3,\zeta^{-51}x_4).$$
\end{exa}

\subsection{Virtual localization formula}
Gromov--Witten (GW) theory of $\PP^N$ and its K-theoretic version is computed by the virtual localization formula under the natural action of the torus $T=\(\CC^*\)^N$.
Unfortunately, there are in general no non-trivial torus-actions preserving a smooth degree-$d$ hypersurface $X$,
leading to many difficulties in the computation of its GW and KGW invariants.
Nevertheless, we have an action of the finite group $G$ on $X$.

In cohomology or in Chow theory, the action of the finite group $G$ is useless with respect to the localization formula.
Indeed, we have e.g.~$A^G_*(\mathrm{pt})=\CC$.
On the other hand, we have $K^0(G,\mathrm{pt}) = R(G)$, the representation ring of the group $G$.
Moreover, there exists in K-theory a (virtual) localization formula under finite group actions.

Unfortunately, the (virtual) localization formula does not give a result in $K^0(G,\mathrm{pt})$, but in a localized ring where we invert equivariant parameters.
For instance, in the case of an abelian group $G = \ZZ/M\ZZ$, the representation ring (taken with complex coefficients) is
$$R(G)_\CC \simeq \CC[X]/(1-X^M),$$
and the multiplicative set we use for localization is
$$\left\lbrace 1-X,\dotsc,1-X^{M-1} \right\rbrace.$$
As a consequence, the localized ring is isomorphic to $\CC$ and the map $R(G)_\CC \to R(G)_{\CC,\loc}$ is not injective.
Precisely, the map sends $X$ to a primitive $M$-th root of unity $\zeta$, so that for every prime divisor $p$ of $M$, the polynomial
$$\sum_{k=0}^p X^{kM/p} \mapsto \sum_{k=0}^p \zeta^{kM/p} = 0.$$
As a conclusion, the (virtual) localization formula successfully computes a $G$-equivariant K-class expressed using roots of unity, but we cannot extract the `non-equivariant' limit corresponding to the map
$$K^0(G,\mathrm{pt}) \simeq \CC[X]/(1-X^M) \to \CC \simeq K^0(\mathrm{pt}) ~,~~ X \mapsto 1.$$

Nevertheless, we find a way to extract some information.
Indeed, K-theoretic invariants have another important feature: they are integers.
Therefore, when the order of the group is a prime number $p$, the defect of injectivity of the map $R(G) \to R(G)_\loc$ amounts to the uncertainty
$$1+X+\dotsb+X^{p-1},$$
which equals $p$ in the non-equivariant limit $X \mapsto 1$.
To conclude, we are left with the desired integer modulo $p$.
Furthermore, if we have several finite actions of different prime orders, we can increase our knowledge about the result.

Let us go back to the degree-$d$ hypersurface $X \subset \PP^N$.
The action of $G = \ZZ/M\ZZ$ on $\PP^N$ leaving $X$ invariant induces a $G$-action on the moduli spaces of stable maps to $\PP^N$ and to $X$, so that their virtual structure sheaves are $G$-equivariant, namely 
$$\cO^\vir_{\cM_g(\PP^N,\beta)} \in G_0(G,\cM_g(\PP^N,\beta)) \quad \textrm{and} \quad \cO^\vir_{\cM_g(X,\beta)} \in G_0(G,\cM_g(X,\beta)).$$
By the virtual localization formula, we then obtain
$$\cO^\vir_{\cM_g(X,\beta)} = \iota_! \left( \cfrac{\cO^\vir_{\cM_g(X,\beta)^\fix}}{\lambda_{-1} ({N^\vir_\iota}^\vee)} \right) \in G_0(G,\cM_g(X,\beta))_\loc,$$
where $\iota \colon \cM_g(X,\beta)^\fix \hookrightarrow \cM_g(X,\beta)$ denotes the $G$-fixed locus, and $N^\vir_\iota$ denotes the moving part of the perfect obstruction theory on the fixed locus.
At last, we get
\begin{equation*}
\chi \left(\cO^\vir_{\cM_g(X,\beta)}\right) = \chi \left( \cfrac{\cO^\vir_{\cM_g(X,\beta)^\fix}}{\lambda_{-1} ({N^\vir_\iota}^\vee)} \right) \in \CC.
\end{equation*}
The next step is to use Theorem \ref{quantum Lefschetz} to relate this formula to a formula for $\PP^N$, where an explicit localization formula is available via the torus action.

\subsection{Fixed locus}
We easily check all the conditions listed in Section \ref{sec enh}, but the second:
\begin{itemize}
\item every stable component of a fixed stable map is contracted.
\end{itemize}
We are able to prove it under the following restrictions on genus of the source curve and degree of the map.

\begin{pro}\label{fixed locus}
Let $G = \ZZ/M\ZZ$ act on a smooth projective variety $X$ such that, for every non-zero element $h \in G$, the $h$-fixed locus $X^h$ consists of isolated points in $X$.
Let $f \colon C \to X$ be a stable map corresponding to a $G$-fixed point in the moduli space $\cM_{g,n}(X,\beta)$.
We assume
\begin{equation}\label{cond}
g < \cfrac{p-1}{2} \quad \textrm{and} \quad \beta<M,
\end{equation}
where $p$ is the greatest prime divisor of $M$.
Then every stable component of the curve $C$ is mapped to one of the $G$-fixed point in $X$.
\end{pro}

\begin{proof}
First, we claim that if $f \colon C \to X$ is a $G$-fixed stable map of positive degree, then the group $G$ is a subgroup of the group $\Aut(C)$ of automorphisms of $C$.
Indeed, let $\zeta \in G$ be a primitive element. Since the stable map is fixed, we can choose an automorphism $\phi_1 \in \Aut(C)$ of the curve $C$ such that
$$\forall x \in C ~,~~ \zeta \cdot f(x) = f(\phi_1(x)).$$
We then define $\phi_k := \phi_\zeta^k$ for $0 \leq k <M$.
Since the degree of the map $f$ is positive and all points in $X$ except a finite number of them are not fixed by any element of $G$, we can choose a point $x \in C$ such that the points
$$\zeta^k \cdot f(x) = f(\phi_k(x))$$
are all distinct for $0 \leq k < M$.
It means that the points $\phi_k(x)$ and thus the elements $\phi_k$ are all distinct.
	
Secondly, let us assume $C$ is a stable curve and is not contracted.
Since $G$ is a subgroup of $\Aut(C)$, then the prime number $p$ divides the order of $\Aut(C)$.
By \cite[Prop. 3.6]{LiuXu}, we get $p \leq 2g+1$, which is a contradiction.

At last, we consider the case where $C$ is not a stable curve (and therefore the degree of the map is positive).
Let $\Gamma_f$ be a dual graph representing the stable map $f$, where we represent every stable component of $C$ by a vertex and every unstable component by an edge. Furthermore, we add on the graph labels to keep track of the genus, number of markings, and degree.
It is clear that every automorphism $\phi_k$ of the curve $C$ induces an automorphism of the dual graph $\Gamma_f$.
Moreover, for each stable component $D$ of $C$ whose corresponding vertex is fixed in $\Gamma_C$, the restriction $f_{|D} \colon D \to X$ is fixed by the group $G$.
	
We aim to show that the set $V_{>0}$ of vertices with positive degree is empty.
Assume it is not.
Then, if the group $G$ acts on $V_{>0}$ without fixed points, the total degree of the map is at least $M$, which is a contradiction.
Therefore, there is at least one fixed-point, i.e.~there exists a stable component $D$ of $C$ such that the restriction $f_{|D}$ is $G$-fixed.
As we have seen above, the stable component $D$ is then contracted to a point, which is a contradiction with the fact that its corresponding vertex is in $V_{>0}$.
\end{proof}

\begin{rem}\label{fixed locus improved}
Proposition \ref{fixed locus} holds as well if the condition $g<\cfrac{p-1}{2}$ is replaced by `for every stable curve of genus less than $g$, there is no automorphism of order equal to $M$'.
\end{rem}

\subsection{Equivariant and congruent formulas}
Let us apply Theorem \ref{quantum Lefschetz} to our situation.
\begin{thm}\label{big theo}
Let $g, n$, and $\beta$ be non-negative integers.
Let $X$ be a degree-$d$ loop hypersurface in $\PP^N$ and take a subgroup $H \subset \ZZ/M\ZZ$ of order $q$ acting on $X$ via the action
\begin{equation}\label{action}
\forall k \in H \subset \ZZ/M\ZZ ~,~~ k \cdot (x_0,\dotsc,x_N) = (x_0, \zeta^k x_1, \zeta^{k \cdot u_2}x_2,\dotsc,\zeta^{k \cdot u_N}x_N),
\end{equation}
see Section \ref{loop}.
This action depends on the choice of a primitive $q$-th root of unity $\zeta$.
Moreover, we have the usual $T:=(\CC^*)^N$-action on $\PP^N$ and we see that it extends the $H$-action via the embedding $\varphi \colon H \hookrightarrow T:=(\CC^*)^N$ sending $k$ to $(\zeta^k x_1, \zeta^{k \cdot u_2}x_2,\dotsc,\zeta^{k \cdot u_N})$.

Assume the following bounds
$$g < \cfrac{p-1}{2} \quad \textrm{and} \quad \beta<q,$$
where $p$ is the greatest prime divisor of $q$.
Denote $A := \bigotimes_{i=1}^n \Psi_i^{a_i} \otimes \ev^*(Y_i)$ some insertions of Psi-classes and K-classes $Y_i \in K^0(T,\PP^N)$ coming from the ambient space.
Then the corresponding $H$-equivariant K-theoretic GW invariant equals
\begin{equation*}
\chi^T \( \lambda_{-1}^T\(R\pi_*f^*\cO(d)\)^\vee \otimes A \otimes \cO_{\cM_{g,n}(\PP^N,\beta)}^{\vir,T} \) \to \chi^H \left(A \otimes \cO^{\vir,H}_{\cM_{g,n}(X,\beta)}\right) \in \CC.
\end{equation*}
Precisely, the class $\lambda_{-1}^T (R \pi_*f^*\cO(d))^\vee$ is only defined after localization, so we first apply the virtual localization formula to the left-hand side, then we compute it in $K^0(T,\overline{\cM}_{g,n})_\loc = K^0(\overline{\cM}_{g,n}) \otimes \CC[\![t_1^{\pm 1},\dotsc,t_N^{\pm 1}]\!]$ as a formal series in the $T$-equivariant parameters and their inverse, then we specialize them to $(t_1,\dotsc,t_N) = (\zeta,\zeta^{u_2},\dotsc,\zeta^{u_N})$ using $\varphi \colon H \hookrightarrow T$ and obtain a well-defined K-class in $K^0(\overline{\cM}_{g,n}) \otimes \CC$. Eventually, we take its Euler characteristic and land in $R(H)_{\CC,\loc} \simeq \CC$, where the last isomorphism depends on the primitive $M$-th root of unity $\zeta$.
\end{thm}

\begin{rem}
The localization map $R(H) \to R(H)_{\CC,\loc}$ corresponds to the map $\ZZ[X]/(1-X^q) \to \CC$ sending the variable $X$ to $\zeta$.
\end{rem}

\begin{rem}
In Theorem \ref{big theo}, it is important that for every non-zero element $h \in H$, the $h$-fixed locus consists of coordinate points in $\PP^N$. It is guaranteed by Proposition \ref{fixloccond} and the fact that $H \subset \ZZ/M\ZZ$.
\end{rem}

\begin{cor}\label{non equiv limit}
We take the same notations and assumptions as in Theorem \ref{big theo}.
We further assume that the order $q$ of the group $H$ is a prime number.
For each $1 \leq k <q$, denote by $B_k \in \CC$ the result of Theorem \ref{big theo} when $\zeta = e^{\frac{2 \ci k \pi}{q}}$.
Then the K-theoretic GW invariant of $X$ equals
$$\chi \left(A \otimes \cO^{\vir}_{\cM_{g,n}(X,\beta)}\right) \equiv -\(B_1+\dotsb+B_{q-1}\) \in \ZZ/q\ZZ.$$
\end{cor}

\begin{proof}
The $H$-equivariant Euler characteristic $\chi^H \left(A \otimes \cO^{\vir,H}_{\cM_{g,n}(X,\beta)}\right)$ lies in the representation ring $R(H) \simeq \ZZ[X]/(1-X^q)$, so that 
there exist integers $a_0,\dotsc,a_{q-1}$ such that
$$\chi^H \left(A \otimes \cO^{\vir,H}_{\cM_{g,n}(X,\beta)}\right) = \sum_{l=0}^{q-1} a_l X^l \in \ZZ[X]/(1-X^q).$$
Our goal would be to compute
$$\chi \left(A \otimes \cO^{\vir}_{\cM_{g,n}(X,\beta)}\right) = \sum_{l=0}^{q-1} a_l \in \ZZ,$$
but Theorem \ref{big theo} only gives us
$\sum_{l=0}^{q-1} a_l \zeta^l \in \CC$.
However, since $q$ is a prime number, we can apply Theorem \ref{big theo} to every primitive $q$-th root of unity $\zeta^k$, $1 \leq k <q$.
Summing the various results, we obtain
$$\sum_{k=1}^{q-1} \sum_{l=0}^{q-1} a_l \zeta^{kl} = q a_0 - \sum_{l=0}^{q-1} a_l \in \ZZ,$$
leading to the congruence.
\end{proof}

\begin{rem}\label{impr}
Assume the order $q$ of the group $H$ is not a prime number and choose a non-zero element $h \in H$.
Even when $h$ is not a primitive element, we can apply Theorem \ref{big theo} to the subgroup $\langle h \rangle$, but we then have the following bounds
$$g < \cfrac{p-1}{2} \quad \textrm{and} \quad \beta < \mathrm{ord}(h),$$
where $\mathrm{ord}(h)$ denotes the order of the element $h$ and $p$ is its greatest prime divisor.
In order to obtain the KGW invariant in $H$, we then need to sum all the results of Theorem \ref{big theo} for all non-zero elements $h \in H$.
Therefore, we have to restrict to the bounds
$$g < \cfrac{p-1}{2} \quad \textrm{and} \quad \beta <p,$$
where $p$ is the smallest prime divisor of $q$.
\end{rem}

\begin{exa}
For the quintic threefold of Example \ref{quint exa}, the specialization of equivariant parameters corresponding to $G \hookrightarrow T$ is
$$(t_0,\dotsc,t_4) = (1,\zeta,\zeta^{-3},\zeta^{13},\zeta^{-51}) ~,~~ \zeta^{205}=1.$$
Moreover, we have a subgroup $H := \ZZ/41\ZZ \subset \ZZ/205\ZZ$, so that by Corollary \ref{non equiv limit}, we are able to compute all KGW invariants modulo $41$ up to genus $19$ and degree $40$.
Moreover, by Remark \ref{impr}, we are able to compute all KGW invariants modulo $205$ in genera $0$ and $1$ up to degree $4$.
\end{exa}

\begin{rem}\label{modfive}
Another way to realize the quintic hypersurface in $\PP^4$ is
$$X = \left\lbrace x_0^5+\dotsb+x_4^5 = 0 \right\rbrace \subset \PP^4.$$
Then the group is $\(\ZZ/5\ZZ\)^4$, but to ensure that the $g$-fixed locus consists of isolated points for every element $g$ of the group, we need to consider the subgroup $G=\ZZ/5\ZZ$, acting as
$$\zeta \cdot \underline{x} = (x_0,\zeta x_1,\zeta^2x_2,\zeta^3x_3,\zeta^4x_4).$$
Furthermore, we observe the $G$-fixed locus is empty.
We then deduce that all KGW invariants in genera $0$ and $1$ and up to degree $4$ vanish modulo $5$.
\end{rem}

\subsection{Example of the quintic threefold}
We illustrate Theorem \ref{big theo} and Corollary \ref{non equiv limit} by a computation of the genus-one degree-one unmarked KGW invariant in the case of the quintic hypersurface in $\PP^4$, modulo $205$.

\begin{pro}\label{exa}
Let $X \subset \PP^4$ be a smooth quintic hypersurface.
We find
$$\chi \left(\cfrac{\cO^{\vir}_{\cM_{1,0}(X,1)}}{1-q \EE^\vee}\right) \equiv (120 q^2+180 q +125) \cfrac{(1-q^4-q^6)}{(1-q^4)(1-q^6)} \in \ZZ/205\ZZ[\![q]\!].$$
\end{pro}

In order to prove Proposition \ref{exa}, we first write the general graph sum formula coming from torus-localization and we then specialize to $(g,n,\beta) = (1,0,1)$.

Following the general scheme of Theorem \ref{big theo}, we compute the K-theoretic class
$$\chi^T \( \lambda_{-1}^T\(R\pi_*f^*\cO(5)\)^\vee \otimes A \otimes \cO_{\cM_{g,n}(\PP^4,\beta)}^{\vir,T} \) \in K^0(\overline{\cM}_{g,n}) \otimes \CC[\![t_0^{\pm 1},\dotsc,t_4^{\pm 1}]\!].$$
It is done via the standard virtual localization formula of \cite{GraberPandha}, lifted to K-theory, as a sum over dual graphs.
Indeed, the class $\lambda_{-1}$ is multiplicative in K-theory, just as the Euler class in cohomology, so that the whole proof of \cite[Section 4]{GraberPandha} holds.
Therefore, we take the same notations as in \cite{GraberPandha}, to which we refer for instance for the description of graphs, except that we take the convention $t_j = -\lambda_j$ with respect to their $T$-weights.

Let $\Gamma$ be a graph in the localization formula of $\PP^4$.
We denote by $\overline{\cM}_\Gamma$ the associated moduli space of stable curves and by $\mathrm{A}_\Gamma$ the group of automorphisms coming from the graph $\Gamma$ and from degrees of the edges, so that the corresponding fixed locus in $\cM_{g,n}(\PP^4,\beta)$ is the quotient stack $[\overline{\cM}_\Gamma/\mathrm{A}_\Gamma]$, see \cite{GraberPandha}.
The contribution of the graph $\Gamma$ to the localization formula is of the form
$$\chi \([\overline{\cM}_\Gamma/\mathrm{A}_\Gamma] ;\Contr(\mathrm{flags}) \cdot \Contr(\mathrm{vertices}) \cdot \Contr(\mathrm{edges}) \),$$
where we have
\begin{eqnarray*}
\Contr(\mathrm{flags}) & = & \cfrac{1}{\(1-\(\cfrac{t_{i(F)}}{t_{j(F)}}\)^{1/d_e} \Psi_F\)} \cdot \cfrac{\prod_{m \neq i(F)} 1-\cfrac{t_{i(F)}}{t_m}}{1-\cfrac{t_{i(F)}^5}{t_0^4t_1}}, \\
\Contr(\mathrm{vertices}) & = &  \cfrac{1-\cfrac{t_{i(v)}^5}{t_0^4t_1}}{\displaystyle{\sum_{k=0}^{g(v)}} (-1)^k \(\cfrac{t_{i(v)}^5}{t_0^4t_1}\)^k \Lambda^k \EE} \cdot \prod_{m \neq i(v)} \cfrac{\displaystyle{\sum_{k=0}^{g(v)}} (-1)^k \(\cfrac{t_{i(v)}}{t_m}\)^k \Lambda^k \EE}{1-\cfrac{t_{i(v)}}{t_m}}, \\
\Contr(\mathrm{edges}) & = & \prod_{a=1}^{d_e} \(2- \(\cfrac{t_{j'}}{t_j}\)^{a/d_e}-\(\cfrac{t_j}{t_{j'}}\)^{a/d_e}\)^{-1} \cdot \prod_{\substack{a+b=d_e \\ k \neq j,j'}} \(1-\cfrac{\(t_j^a t_{j'}^b\)^{1/d_e}}{t_k}\)^{-1} \\
& &  \cdot \prod_{a+b=5 d_e} \(1-\cfrac{\(t_j^a t_{j'}^b\)^{1/d_e}}{t_0^4t_1}\),
\end{eqnarray*}
where we write here the contribution of an edge linking the coordinate points $p_j$ and $p_{j'}$.

\begin{rem}
In the contribution of vertices, we can rewrite the sum in terms of the lambda-structure as
$\lambda_{-u}(\EE)$, with $u:=\cfrac{t_{i(v)}^5}{t_0^4t_1}$.
\end{rem}

\begin{rem}
All contributions are in $K^0(\overline{\cM}_\Gamma) \otimes \CC[\![t_0^{\pm 1/d},\dotsc,t_4^{\pm 1/d}]\!]_{d \in \NN}$.
However, taking the K-theoretic pushforward map from $[\overline{\cM}_\Gamma/\mathrm{A}_\Gamma]$ to a point corresponds to extracting the $\mathrm{A}_\Gamma$-invariant part in $K^0_T(\overline{\cM}_\Gamma)$, which means that all summands with non-integral powers of $t_0,\dotsc,t_4$ become zero.
\end{rem}

Let us now specialize the formula to $(g,n,\beta) = (1,0,1)$.
The graph $\Gamma$ has only two vertices $v_1$ and $v_2$, of respective genera $1$ and $0$, and one degree-one edge in between.
Moreover, as the vertex $v_2$ has valence one, it corresponds to a free point (not marked, not a node) rather than to a stable component of the curve.
We denote by $0 \leq i_1 \neq i_2 \leq 4$ the indices of the coordinate points $p_{i_1}$ and $p_{i_2}$ to which the vertices $v_1$ and $v_2$ are sent by the stable map.
Note also that such a graph has no automorphisms and the moduli space $\overline{\cM}_\Gamma$ is isomorphic to $\overline{\cM}_{1,1}$.
Furthermore, we recall that the Hodge bundle $\EE$ over $\overline{\cM}_{1,1}$ is identified with the cotangent line $\Psi_1$ at the marking.
As a consequence, the virtual localization formula equals
\begin{eqnarray*}
\sum_{0 \leq i_1 \neq i_2 \leq 4} && \cfrac{\prod_{a+b=5} \(1-\cfrac{t_{i_1}^a t_{i_2}^b}{t_0^4t_1}\)}{
\(2- \cfrac{t_{i_1}}{t_{i_2}}-\cfrac{t_{i_2}}{t_{i_1}}\) \cdot \prod_{k \neq i_1,i_2} \(1-\cfrac{t_{i_1}}{t_k}-\cfrac{t_{i_2}}{t_k}+\cfrac{t_{i_1}t_{i_2}}{t_k^2}\)} \\
&& \chi \(\overline{\cM}_{1,1} ; \cfrac{1}{1-q \EE^\vee} \cfrac{\displaystyle{\prod_{m \neq i_1,i_2}} \(1- \cfrac{t_{i_1}}{t_m} \Psi_1\)}{1-\cfrac{t_{i_1}^5}{t_0^4t_1} \Psi_1}\).
\end{eqnarray*}

Once we specialize to $(t_0,\dotsc,t_4) = (1,\zeta,\zeta^{-3},\zeta^{13},\zeta^{-51})$, where $\zeta$ is any primitive root of unity of order $41$, we notice that denominators never vanish, but the numerator could vanish, precisely
\begin{eqnarray*}
1-\cfrac{t_{i_1}^a t_{i_2}^b}{t_0^4t_1} = 0  \iff && i_2=i_1+1 \textrm{ and } (a,b)=(4,1) \\
& \textrm{or }&
i_1=i_2+1 \textrm{ and } (a,b)=(1,4),
\end{eqnarray*}
with the cyclic convention on indices, i.e.~$t_5:=t_0$.
Moreover, we have
$$1- \cfrac{t_{i_1}}{t_{i_1+1}} \Psi_1 = 1-\cfrac{t_{i_1}^5}{t_0^4t_1} \Psi_1,$$
so that the specialization of the localization formula gives
\begin{eqnarray*}
\sum_{\substack{0 \leq i_1 \neq i_2 \leq 4 \\ i_2 \neq i_1+1 \\ i_1 \neq i_2+1}} && \left[\cfrac{\prod_{a+b=5} \(1-\cfrac{t_{i_1}^a t_{i_2}^b}{t_0^4t_1}\)}{
\(2- \cfrac{t_{i_1}}{t_{i_2}}-\cfrac{t_{i_2}}{t_{i_1}}\) \cdot \prod_{k \neq i_1,i_2} \(1-\cfrac{t_{i_1}}{t_k}-\cfrac{t_{i_2}}{t_k}+\cfrac{t_{i_1}t_{i_2}}{t_k^2}\)} \right. \\
&& \left. \chi \(\overline{\cM}_{1,1} ; \cfrac{1}{1-q \EE^\vee} \displaystyle{\prod_{m \neq i_1,i_1+1,i_2}} \(1- \cfrac{t_{i_1}}{t_m} \Psi_1\)\) \right]_{(t_0,\dotsc,t_4) = (1,\zeta,\zeta^{-3},\zeta^{13},\zeta^{-51})}.
\end{eqnarray*}

By \cite[Prop. 2.9]{LeeQu}, we have
\begin{eqnarray*}
&& \chi \(\overline{\cM}_{1,1} ; \cfrac{1}{1-q \EE^\vee} \cfrac{1}{1-q_1 \Psi_1} \) = \\
&& \cfrac{(1-qq_1)(1-q^4-q^6-q_1^2q^6-q_1^2q^8-q_1^4q^8+q^2q_1^2+q^4q_1^4+q^6q_1^6+q^8q_1^8)}{(1-q^4)(1-q^6)(1-q_1^4)(1-q_1^6)},
\end{eqnarray*}
hence we get
\begin{eqnarray*}
\chi \(\overline{\cM}_{1,1} ; \cfrac{\displaystyle{\prod_{m \neq i_1,i_1+1,i_2}} \(1- \cfrac{t_{i_1}}{t_m} \Psi_1\)}{1-q \EE^\vee} \) & = & \cfrac{(1-q^4-q^6)}{(1-q^4)(1-q^6)} \displaystyle{\prod_{m \neq i_1,i_1+1,i_2}} \(1+ \cfrac{t_{i_1}}{t_m} q\).
\end{eqnarray*}

As a consequence, our formula simplifies as
\begin{eqnarray*}
\cfrac{(1-q^4-q^6)}{(1-q^4)(1-q^6)} \sum_{\substack{0 \leq i_1 \neq i_2 \leq 4 \\ i_2 \neq i_1+1 \\ i_1 \neq i_2+1}} \left[\cfrac{\prod_{a+b=5} \(1-\cfrac{t_{i_1}^a t_{i_2}^b}{t_0^4t_1}\) \displaystyle{\prod_{m \neq i_1,i_1+1,i_2}} \(1+ \cfrac{t_{i_1}}{t_m} q\)}{
\(2- \cfrac{t_{i_1}}{t_{i_2}}-\cfrac{t_{i_2}}{t_{i_1}}\) \cdot \prod_{k \neq i_1,i_2} \(1-\cfrac{t_{i_1}}{t_k}-\cfrac{t_{i_2}}{t_k}+\cfrac{t_{i_1}t_{i_2}}{t_k^2}\)} \right]_{(t_0,\dotsc,t_4) = (1,\zeta,\zeta^{-3},\zeta^{13},\zeta^{-51})}.
\end{eqnarray*}
At last, we must take the opposite of the sum of these expressions over all primitive roots $\zeta$ of order $41$.
First, we notice that the term inside the sum is a polynomial in $q$ of degree at most two, so that it is enough to evaluate it at $q \in \left\lbrace 0,1,2 \right\rbrace$.
Using Sagemath, we eventually find
\begin{eqnarray*}
\chi \left(\cfrac{\cO^{\vir}_{\cM_{1,0}(X,1)}}{1-q \EE^\vee}\right) & \equiv & (-85 q^2+590 q -80) \cfrac{(1-q^4-q^6)}{(1-q^4)(1-q^6)} \\
& \equiv & (38 q^2+16 q +2) \cfrac{(1-q^4-q^6)}{(1-q^4)(1-q^6)} \in \ZZ/41\ZZ[\![q]\!].
\end{eqnarray*}
Furthermore, using Remark \ref{modfive}, we obtain the result of Proposition \ref{exa}.

\subsection{Special case of elliptic curves}
In this section, we use the ideas behind Corollary \ref{non equiv limit} to prove that KGW theory with homogeneous insertions of an elliptic curve is trivial.

\begin{pro}\label{elliptic}
Let $E$ be an elliptic curve.
Then for every genus $g$ and degree $\beta$, number of markings $n$, and insertions $A := \bigotimes_{i=1}^n \Psi_i^{a_i} \otimes Y_i$, with $2g-2+n>0$ and $Y_i \in K^0(E)$ homogeneous K-classes,
the corresponding KGW invariant vanishes
$$\chi \left(A \otimes \cO^{\vir}_{\cM_{g,n}(E,\beta)}\right)=0.$$
\end{pro}

\begin{proof}
Let $M$ be the largest possible order of an automorphism of a stable curve of genus $g$.
Let $p$ be any prime number larger than $M+1$ and $\beta+1$.
Define $G:=\ZZ/p\ZZ$ and take a $G$-torsion point $x \in E$.
Then the group $G$ acts on the elliptic curve $E$ by translation $y \mapsto y+x$, and for every non-zero element $h \in G$, the $h$-fixed locus is empty.
By Remark \ref{fixed locus improved} and Proposition \ref{fixed locus}, the $G$-fixed locus in the moduli space of stable maps $\cM_{g,n}(E,\beta)$ is empty.
Therefore, by the localization formula, the $G$-equivariant KGW invariant vanishes, so that we get
$$\chi \left(A \otimes \cO^{\vir}_{\cM_{g,n}(E,\beta)}\right) \equiv 0 \in \ZZ/p\ZZ$$
for the non-equivariant limit.
Since it is true for infinitely many prime numbers $p$, then we obtain the vanishing in $\ZZ$.
\end{proof}

\begin{rem}
Interestingly, KGW invariants are deduced from GW invariants via a Kawazaki--Riemann--Roch theorem, see \cite{Tonita} and \cite[Part IX]{Givental3}.
It would be instructive to compare Proposition \ref{elliptic} with GW theory of elliptic curves, which is non-trivial and described in \cite{curve1,curve3}.
\end{rem}
 
\begin{rem}
The same proof holds for abelian varieties.
However, when the dimension of the abelian variety is greater than $2$ and the degree-class $\beta$ is non-zero, there is a trivial quotient of the obstruction theory, so that both GW and KGW theories are trivial.
However, for degree-zero invariants, GW theory is non-trivial, but KGW theory is.
\end{rem}

\begin{rem}
The main idea in the proof of Proposition \ref{elliptic} is to use congruence relations for infinitely many prime numbers.
Indeed, if we are able to find, for a smooth DM stack $X$, automorphisms of prime orders for infinitely many primes and to compute the localization formulae, then we would be able to know all KGW invariants of $X$.
Therefore, a necessary condition is that the automorphism group of $X$ must be infinite.
However, it is not sufficient. For instance, some K3 surfaces have infinitely many symmetries, but it was shown by \cite{Kondo} that the maximal order of a finite group acting faithfully on a K3 surface is $3840$.
\end{rem}

\begin{rem}
Here are a few remarks on finiteness of automorphism group.
For projective hypersurfaces (except quadrics, elliptic curves, and K3 surfaces), every automorphism is projective and the automorphism group is finite.
All Batyrev Calabi--Yau (CY) $3$-folds have finite automorphism groups, see \cite{Batyrev}.
Every projective variety of general type has finite automorphism group.
CY varieties with Picard numbers $1$ or $2$ have finite automorphism groups. It is expected that most CY varieties with Picard number more than $4$ have infinitely many automorphisms.
In particular, it would be interesting to know whether the Schoen CY $3$-fold has automorphisms of prime order for infinitely many primes and to study its KGW theory, see \cite{Machine}.
\end{rem}

\section{K-theoretic FJRW theory}\label{FJRWth}
Similarly to KGW theory, we aim in this section to compute the K-theoretic FJRW invariants of a Landau--Ginzburg (LG) orbifold modulo prime numbers.
For simplicity of the exposition, we focus in this paper on the quintic polynomial with minimal group of symmetries.
However, it is straightforward to apply the same ideas to an LG orbifold $(W,G)$, where $W$ is an invertible polynomial and $G$ is an admissible group, as long as we only insert $\Aut(W)$-invariant states in the correlator. We refer to \cite{Guere1} for details.

\subsection{Sketch of Polishchuk--Vaintrob construction}
Let $W(x_1,\dotsc,x_5)$ be a quintic polynomial in five variables and let $\mu_5$ act on $\CC^5$ by multiplication by a fifth-root of unity.
The moduli space used in FJRW theory of $(W,\mu_5)$ is the moduli space $\sS_{g,n}^{1/5}$, which parametrizes $(\cC,\sigma_1,\dotsc,\sigma_n,\cL,\phi)$.
Precisely, the curve $\cC$ is an orbifold genus-$g$ stable curve with isotropy group $\mu_5$ at the markings $\sigma_1,\dotsc,\sigma_n$ and at the nodes (and trivial everywhere else), $\cL$ is a line bundle on $\cC$, and $\phi \colon \cL \to \omega_{\log}:=\omega_\cC(\sigma_1+\dotsb+\sigma_n)$ is an isomorphism.

Let $\pi$ be the projection from the universal curve to $\sS_{g,n}^{1/5}$ and $\cL$ be the universal line bundle.
In \cite{Polish1}, Polishchuk and Vaintrob constructed resolutions $R\pi_*\(\cL^{\oplus 5}\)=[A \to B]$ by vector bundles over $\sS_{g,n}^{1/5}$ such that there exists some morphism
$$\alpha \colon \cSy^4 A \to B^\vee$$
corresponding to the differentiation of the polynomial $W$, see \cite{Guere1} for details.
Taking $p \colon X \to \sS_{g,n}^{1/5}$ to be the total space of the vector bundle $A$, then the morphism $\alpha$ is interpreted as a global section of $p^*B^\vee$ over $X$, and the map $\beta \colon A \to B$ coming from the resolution is interpreted as a global section of $p^*B$.
As a consequence, we obtain a Koszul matrix factorization
$$\PV := \left\lbrace \alpha,\beta \right\rbrace := \(\Lambda^\bullet B^\vee,\alpha \wedge \cdot + \iota_{\beta}\) \in D(X,\alpha(\beta))$$
of potential $\alpha(\beta)$ over the space $X$, and the support of this matrix factorization is exactly the moduli space $\sS_{g,n}^{1/5}$.

The moduli space $\sS_{g,n}^{1/5}$ has several components depending on the monodromies $\underline{\gamma}:=(\gamma_1,\dotsc,\gamma_n) \in \mu_5^n$ at the markings, we denote by $\sS_{g,n}^{1/5}(\underline{\gamma})$ the corresponding component.
Assume all monodromies are non-zero, this is known as the \textit{narrow condition}.
Then the pairing $\alpha(\beta)$ is the zero function over $X$, and the matrix factorization $\PV$ becomes a two-periodic complex, exact off the moduli space $\sS_{g,n}^{1/5}(\underline{\gamma})$.
Therefore, we can define the push-forward along the projection map $p$ in the category of matrix factorization, yielding
$$p_*(\PV) \in D(\sS_{g,n}^{1/5}(\underline{\gamma}),0) \simeq D^b(\sS_{g,n}^{1/5}(\underline{\gamma})),$$
where on the right we have the derived category of coherent sheaves.

\begin{rem}\label{broad}
If we allow trivial monodromies (i.e.~we consider broad insertions), then the pairing $\alpha(\beta)$ does not vanish and we rather end with a functor
\begin{equation*}
\begin{array}{lccccl}
\Phi \colon & D_\Gamma(\aA^{\overline{\gamma}},W_{\overline{\gamma}}) & \longrightarrow & D_{\overline{G}}(\sS^\mathrm{rig}_{g,n}(\overline{\gamma}),0) & \simeq D([\sS^\mathrm{rig}_{g,n}(\overline{\gamma})/{\overline{G}}]) \\
&                               U                            & \longmapsto     & p_*(\ev^*(U) \otimes \PV),  & \\
\end{array}
\end{equation*}
where we need to consider rigidified moduli spaces, see \cite{Polish1} for details.
\end{rem}
	
In general, to any triangulated category $\cC$, we associate a Grothendieck group $K_0(\cC)$ by taking the free abelian group generated by the objects of the category and then moding out the relation
$$[A] - [B] + [C] = 0$$
for every distinguished triangle $A \to B \to C$.
Furthermore, any functor $f \colon \cC_1 \to \cC_2$ of triangulated dg categories induces a morphism of groups
$$f_! \colon K_0(\cC_1) \to K_0(\cC_2).$$
Eventually, when the category is the derived category of coherent sheaves on a smooth DM stack, we recover the usual K-theory of the stack.

\begin{rem}
If we apply it to the functor of Remark \ref{broad}, we get a morphism of groups
$$\Phi_! := K_0(\Phi) \colon K_0(D_\Gamma(\aA^{\overline{\gamma}},W_{\overline{\gamma
}})) \to K^0([\sS^\mathrm{rig}_{g,n}(\overline{\gamma})/\overline{G}]).$$
\end{rem}

\begin{dfn}\label{K-class}
We define the K-theoretic class
$$\cO^\vir_{\sS_{g,n}^{1/5}(\underline{\gamma})} := p_*(\PV) \in K^0(\sS_{g,n}^{1/5}(\underline{\gamma}))$$
and we call it the virtual structure sheaf of the moduli space $\sS_{g,n}^{1/5}(\underline{\gamma})$.
\end{dfn}

\begin{dfn}
Let $d_1,\dotsc,d_n \in \ZZ$ and non-trivial monodromies $\gamma_1,\dotsc,\gamma_n \in \mu_5$.
The K-theoretic FJRW invariant of the LG orbifold $(W,\mu_5)$ is
$$\chi\(\Psi_1^{\otimes d_1} \otimes \dotsm \otimes \Psi_n^{\otimes d_n} \otimes \cO^\vir_{\sS_{g,n}^{1/5}(\underline{\gamma})} \) \in \ZZ,$$
where $\chi \colon K^0(\sS_{g,n}^{1/5}(\underline{\gamma})) \to K^0(\mathrm{pt}) = \ZZ$ is the Euler characteristic and the line bundle $\Psi_i$ is the relative cotangent line at the $i$-th marked point.
\end{dfn}

\begin{rem}
A special feature of the LG orbifold $(W,\mu_5)$, and more generally when the polynomial has degree $d$ and the group is $\mu_d$, is that FJRW invariants do not depend on the polynomial $W$, as long as it is non-degenerate with the same weights and degree.
The same result holds for the K-theoretic version, as it holds at the matrix factorization level.
Hence, we can consider several choices for our quintic polynomial.
\end{rem}

\subsection{Invertible polynomials}
In the context of mirror symmetry, a well-behaved class of polynomials has been introduced by Berglund--H\"ubsch \cite{Hubsch}.
We say that a polynomial is invertible when it is non-degenerate and with as many variables as monomials.
According to Kreuzer--Skarke \cite{KS}, every invertible polynomial is a (Thom--Sebastiani) sum of invertible polynomials, with disjoint sets of variables, of the following three types
\begin{equation}\label{ThomSebastiani}
\begin{array}{lll}
\textrm{Fermat:} & \qquad x^{a+1} & \\
\textrm{chain of length } c: & \qquad x_1^{a_1}x_2+\dotsb+x_{c-1}^{a_{c-1}} x_c+x_c^{a_c+1} & (c \geq 2), \\
\textrm{loop of length } l: & \qquad x_1^{a_1}x_2+\dotsb+x_{l-1}^{a_{l-1}} x_l+x_l^{a_l}x_1 & (l \geq 2). \\
\end{array}
\end{equation}
In the case of the quintic polynomial, we have for example the following choices:
\begin{eqnarray}\label{quintic pol}
\textrm{Fermat:} & x_1^5+x_2^5+x_3^5+x_4^5+x_5^5 & \(\mu_5\)^5, \nonumber \\
\textrm{loop:} & x_1^4x_2+x_2^4x_3+x_3^4x_4+x_4^4x_5+x_5^4x_1 & \mu_{1025}, \nonumber \\
\textrm{chain:} & x_1^4x_2+x_2^4x_3+x_3^4x_4+x_4^4x_5+x_5^5 & \mu_{1280}, \\
\textrm{$2$-loops:} & x_1^4x_2+x_2^4x_3+x_3^4x_1+x_4^4x_5+x_5^4x_4 & \mu_{15} \times \mu_{65}, \nonumber \\
\textrm{loop-Fermat:} & x_1^4x_2+x_2^4x_3+x_3^4x_4+x_4^4x_1+x_5^5 & \mu_{255} \times \mu_5, \nonumber
\end{eqnarray}
where on the right we write the group $\Aut(W)$ of diagonal matrices leaving the polynomial invariant.

Let $W$ be an invertible quintic polynomial and $\Aut(W)$ be its maximal group of diagonal symmetries.
Recall the space $X$ is the total space of the vector bundle $A$ over the moduli space $\sS_{g,n}^{1/5}(\underline{\gamma})$, and that we have
$$R\pi_*\(\cL^{\oplus 5}\) = [A \to B].$$
Therefore, the vector bundles $A$ and $B$ are direct sums of five copies, that we write
$$A = A_1 \oplus \dotsb \oplus A_5 \quad \textrm{and} \quad B = B_1 \oplus \dotsb \oplus B_5.$$
We then have a natural action of $\Aut(W)$ on the vector bundles $A$ and $B$ by rescaling the fibers.
Precisely, the actions on $X$ in the examples \eqref{quintic pol} are
\begin{eqnarray}\label{quintic act}
\textrm{Fermat:} & (\zeta_1 x_1, \zeta_2 x_2,\zeta_3 x_3,\zeta_4 x_4,\zeta_5 x_5) & \zeta_j = e^{\frac{2 \ci \pi}{5}}, \nonumber \\
\textrm{loop:} & (\zeta x_1, \zeta^{-4} x_2,\zeta^{16} x_3,\zeta^{-64} x_4,\zeta^{256} x_5) & \zeta = e^{\frac{2 \ci \pi}{1025}}, \nonumber \\
\textrm{chain:} & (\zeta x_1, \zeta^{-4} x_2,\zeta^{16} x_3,\zeta^{-64} x_4, \zeta^{256} x_5) & \zeta = e^{\frac{2 \ci \pi}{1280}}, \\
\textrm{$2$-loops:} & (\zeta_1 x_1, \zeta_1^{-4} x_2,\zeta_1^{16} x_3,\zeta_2 x_4, \zeta_2^{-4} x_5) & \zeta_1 = e^{\frac{2 \ci \pi}{65}} ~,~~ \zeta_2 = e^{\frac{2 \ci \pi}{15}}, \nonumber \\
\textrm{loop-Fermat:} & (\zeta_1 x_1, \zeta_1^{-4} x_2,\zeta_1^{16} x_3,\zeta_1^{-64} x_4, \zeta_2 x_5) & \zeta_1 = e^{\frac{2 \ci \pi}{255}} ~,~~ \zeta_2 = e^{\frac{2 \ci \pi}{5}}. \nonumber
\end{eqnarray}
By construction, since the polynomial $W$ is $\Aut(W)$-invariant, the matrix factorization $\PV$ is $\Aut(W)$-equivariant and so is the virtual structure sheaf.

However, we need to be careful when we compute the $\Aut(W)$-fixed locus.
Indeed, the automorphism group of a $(W,\mu_5)$-spin curve $(\cC,\sigma_1,\dotsc,\sigma_n,\cL)$ is $\mu_5 \times \(\mu_5\)^{\#\textrm{nodes}}$, where the first factor rescales the line bundle $\cL$ and the second factor acts only on the orbifold curve $\cC$, it is the so-called ghost automorphism.
As a consequence, we rather consider the action of the group
$$G := \Aut(W)/\mu_5$$
on the space $X$ over the moduli space $\sS_{g,n}^{1/5}(\underline{\gamma})$.

Unfortunately, we still have too many fixed points.
For instance, in the Fermat example, the point
$$(\cC,\sigma_1,\dotsc,\sigma_n,\cL;x_1,0,\dotsc,0) \in X$$
is fixed.
Another example is the chain polynomial, for which the point
$$(\cC,\sigma_1,\dotsc,\sigma_n,\cL;0,\dotsc,0,x_5) \in X$$
is fixed.
In both cases, the $G$-fixed locus is non-compact.
We check easily that, among all invertible quintic polynomials, the only cases where the $G$-fixed locus is compact are
\begin{itemize}
\item the loop polynomial with group $G = \Aut(W)/\mu_5 = \mu_{205}$,
\item the $2$-loops polynomial with group $G = \( \mu_{65} \times \mu_{15} \) / \mu_5$.
\end{itemize}
Moreover, the $G$-fixed locus in the space $X$ equals the base $\sS_{g,n}^{1/5}(\underline{\gamma})$.
We are therefore able to apply the (non-virtual) localization formula on the smooth space $X$ to get the following theorem.

\begin{rem}
It is more convenient to work with cyclic groups. Therefore, in the $2$-loops polynomial case, we prefer to use $G=\mu_{195}$, where the $G$-action on $X$ is
$$(\zeta^{15} x_1, \zeta^{-60} x_2,\zeta^{240} x_3,\zeta^{65} x_4, \zeta^{-260} x_5) ~,~~ \zeta = e^{\frac{2 \ci \pi}{195}}.$$
\end{rem}

\begin{dfn}
Let $l \in \ZZ$.
Adams operation $\Psi^l$ in K-theory is defined on a line bundle $L$ over a space $S$ as
$$\Psi^l (L) := L^{\otimes l}$$
and then extended as a ring homomorphism
$$\Psi^l \colon K^0(S) \to K^0(S).$$
\end{dfn}

\begin{thm}\label{equiv FJRW}
Consider the two following situations:
\begin{itemize}
\item $W$ be the loop polynomial, $G:=\mu_{205}$, $\zeta$ be a primitive $205$-th root of unity, and $(a_1,\dotsc,a_5) = (1,-4,16,-64,256)$,
\item $W$ be the $2$-loops polynomial, $G:=\mu_{195}$, $\zeta$ be a primitive $195$-th root of unity, and $(a_1,\dotsc,a_5) = (15,-60,240,65,-260)$.
\end{itemize}
Let $g$ and $n$ be non-negative integers in the stable range $2g-2+n>0$, and $\underline{\gamma} \in \mu_5^n$ be non-trivial monodromies.
Then the $G$-equivariant virtual structure sheaf equals
\begin{equation}\label{equiv formul}
\cO^{\vir,G}_{\sS_{g,n}^{1/5}(\underline{\gamma})} =  \exp \( \sum_{l \leq -1} \sum_{j=1}^5 \cfrac{\zeta^{a_j \cdot l}}{l} \Psi^l (-R\pi_*\cL) \) \in K^0(\sS_{g,n}^{1/5}(\underline{\gamma})) \otimes \CC.
\end{equation}
\end{thm}

\begin{proof}
In the $G$-equivariant K-theory of the space $X$, the matrix factorization equals
$$\PV = \lambda^G_{-1}p^*B^\vee \in K^0(G,X)$$
and by the localization formula we get
$$\PV = \iota_! \( \cfrac{\lambda^G_{-1}B^\vee}{\lambda^G_{-1} A^\vee} \) = \iota_! \( \lambda^G_{-1}\(B^\vee - A^\vee \)\) \in K^0(G,X)_\loc$$
in the localized ring, where $\iota \colon \sS_{g,n}^{1/5}(\underline{\gamma}) \hookrightarrow X$ is the zero section.
Taking the push-forward along the projection map $p$, we obtain the $G$-equivariant virtual structure sheaf
$$\cO^{\vir,G}_{\sS_{g,n}^{1/5}(\underline{\gamma})} = \lambda^G_{-1}\(B^\vee - A^\vee \) \in K^0(G,\sS_{g,n}^{1/5}(\underline{\gamma}))_\loc \simeq K^0(\sS_{g,n}^{1/5}(\underline{\gamma})) \otimes \CC.$$
If $V$ is a vector bundle, we can express the $\lambda$-structure in terms of Adams operators via the formula
\begin{equation*}
\lambda_{-p}(V^\vee) = \exp \left( \sum_{l \leq -1} \cfrac{p^{-l}}{l} ~~ \Psi^l(V) \right),
\end{equation*}
Moreover, if the action of a group $G$ on the vector bundle $V$ is by rescaling fibers with $\zeta \in G$, then
\begin{equation*}
\lambda^G_{-1}(V^\vee) = \lambda_{-\zeta^{-1}}(V^\vee) = \exp \left( \sum_{l \leq -1} \cfrac{\zeta^l}{l} ~~ \Psi^l(V) \right),
\end{equation*}
In our situation, we find formula \eqref{equiv formul}.
\end{proof}

\begin{rem}
In \cite{Guere1}, we define the characteristic class $\fc_t \colon K^0(S) \to H^*(S)[\![t]\!]$ by
$$\fc_t(B-A) = \Ch(\lambda_{-t}(B^\vee - A^\vee)) \Td(B-A)$$
and we then obtain the formula
$$\lim_{t \to 1} \prod_{i=1}^5 \fc_{t_j}(-R\pi_*\cL) = \cvir \in H^*(\sS_{0,n}^{1/5}(\underline{\gamma}))$$
for the FJRW virtual cycle of $(W,\mu_5)$, where $t_j := t^{a_j}$.
This formula is only valid in genus $0$ and we do not expect the left-hand side to converge in positive genus when $t \to 1$.
However, by Theorem \ref{equiv FJRW}, we see that the formula converges for every genus when $t \to \zeta$.
\end{rem}

In order to get congruences for the non-equivariant limit, we need to consider a subgroup of $\Aut(W)$ with prime order and whose fixed locus in $X$ is compact.
The only invertible polynomial for which it is possible is the loop polynomial, together with the subgroup $\mu_{41}$ acting on $X$ as
$$(\zeta_{41} x_1,\zeta_{41}^{37} x_2,\zeta_{41}^{16} x_3,\zeta_{41}^{18} x_4,\zeta_{41}^6 x_5) ~,~~ \zeta_{41} :=e^{\frac{2 \ci \pi}{41}}.$$

\begin{rem}
The prime decomposition of $205$ is $5 \cdot 41$, so that we could also hope for a congruence modulo $5$. However, the subgroup $\mu_5$ acts trivially on $X$. Indeed, it acts as
$$(\zeta_5 x_1,\zeta_5 x_2,\zeta_5 x_3,\zeta_5 x_4,\zeta_5 x_5) ~,~~ \zeta_5 :=e^{\frac{2 \ci \pi}{5}},$$
which is rescaled by the automorphism group of the $(W,\mu_5)$-spin curve, so that the fixed locus is $X$.
\end{rem}

\begin{cor}\label{nonequiv formul}
Let $W$ be the loop polynomial and $(a_1,\dotsc,a_5) = (1,-4,16,-64,256)$.
For any non-negative integers $g$ and $n$ in the stable range $2g-2+n>0$, non-trivial monodromies $\underline{\gamma} \in \mu_5^n$, and integers $d_1,\dotsc,d_n \in \ZZ$, the K-theoretic FJRW invariant of $(W,\mu_5)$
$$\chi\(\Psi_1^{\otimes d_1} \otimes \dotsm \otimes \Psi_n^{\otimes d_n} \otimes \cO^\vir_{\sS_{g,n}^{1/5}(\underline{\gamma})} \) \in \ZZ$$
equals
\begin{equation*}
-\sum_{k=1}^{40}
\chi\(\Psi_1^{\otimes d_1} \otimes \dotsm \otimes \Psi_n^{\otimes d_n} \otimes \exp \( \sum_{l \leq -1} \sum_{j=1}^5 \cfrac{e^{\frac{2 \ci \pi k l \cdot a_j}{41}}}{l} \Psi^l (-R\pi_*\cL) \)\) \in \ZZ/41\ZZ.
\end{equation*}
\end{cor}

\begin{rem}
More generally, the K-class
\begin{equation*}
B_{41} := -\sum_{k=1}^{40} \exp \( \sum_{l \leq -1} \sum_{j=1}^5 \cfrac{e^{\frac{2 \ci \pi k l \cdot a_j}{41}}}{l} \Psi^l (-R\pi_*\cL) \) \in K^0(G,\sS_{g,n}^{1/5}(\underline{\gamma}))
\end{equation*}
lies in the K-theoretic ring with $\ZZ$-coefficients and we know there exists another K-class $R$ in $K^0(G,\sS_{g,n}^{1/5}(\underline{\gamma}))$ with $\ZZ$-coefficients such that
$$\cO^{\vir}_{\sS_{g,n}^{1/5}(\underline{\gamma})} = B_{41}+41 \cdot R \in K^0(G,\sS_{g,n}^{1/5}(\underline{\gamma})).$$
This yields the following formula for the FJRW virtual cycle of $(W,\mu_5)$
\begin{equation*}
\cvir = \sum_{k=1}^{40} \lim_{t \to \zeta_{41}^k} \prod_{i=1}^5 \fc_{t_j}(-R\pi_*\cL) + 41 \cdot \Ch\(R\) \cdot \Td(-R\pi_*\cL)^5 \in H^*(\sS_{g,n}^{1/5}(\underline{\gamma}))_\QQ.
\end{equation*}
However, since the cohomology is taken with $\QQ$-coefficients, we do not see clear effects of the congruences for the virtual structure sheaf.
An idea would be to replace $R$ with an integral linear combination of (natural) vector bundles over $\sS_{g,n}^{1/5}(\underline{\gamma})$ and see if it is possible to tune coefficients so that the virtual cycle is pure-dimensional.
\end{rem}

\begin{rem}
We observe, from its definition using the quintic Fermat polynomial, that the virtual structure sheaf decomposes into five identical summands, each one corresponding to the so-called $5$-spin theory.
It is then divisible by five in the K-theoretic ring with $\ZZ$ coefficients. However, the K-class $R$ is not divisible by five, as it corresponds to the sub two-periodic complex of $p_*(\PV)$ generated by $\cSy^{41 \cdot l} A^\vee_1$, which is not stable under the map $\cO \to \cSy^4 A_1^\vee \otimes B_1^\vee$.
\end{rem}




\begin{bibdiv}
\begin{biblist}








\bib{BF}{article}{
	author={Behrend, Kai},
	author={Fantechi, Barbara},
	title={The intrinsic normal cone},
	journal={Inv. Math},
	volume={128},
	date={1997},
	number={},
	pages={45-88},
} 

\bib{Hubsch}{article}{
	author={Berglund, Per},
	author={H\"ubsch, Tristan},
	title={A generalized construction of mirror manifolds},
	journal={Nuclear Physics B},
	volume={393},
	date={1993},
	number={},
	pages={391-397},
}



\bib{Candelas}{article}{
	author={Candelas, Philip},
	author={de~la Ossa, Xenia~C.},
	author={Green, Paul~S.},
	author={Parkes, Linda},
	title={A pair of {C}alabi--{Y}au manifolds as an exact soluble superconformal theory},
	date={1991},
	journal={Nucl. Phys.},
	volume={B 359},
	pages={21-74},
}

\bib{Li}{article}{
	author={Chang, Huai-Liang},
	author={Li, Jun},
	title={Gromov--Witten invariants of stable maps with fields},
	journal={IMRN},
	volume={18},
	date={2012},
	number={},
	pages={4163-4217},
}

\bib{Li5}{article}{
	author={Chang, Huai-Liang},
	author={Guo, Shuai},
	author={Li, Jun},
	title={BCOV's Feynman rule of quintic 3-folds},
	journal={available at arXiv:1810.00394v2},
	volume={},
	date={},
	number={},
	pages={},
}

\bib{LGCY}{article}{
	author={Chiodo, Alessandro},
	author={Iritani, Hiroshi},
	author={Ruan, Yongbin},
	title={Landau--Ginzburg/Calabi--Yau correspondence, global mirror symmetry and Orlov equivalence},
	journal={Publications math\'ematiques de l'IH\'ES},
	volume={119},
	date={2014},
	number={1},
	pages={127-216},
}

\bib{LG/CYquintique}{article}{
	author={Chiodo, Alessandro},
	author={Ruan, Yongbin},
	title={Landau--Ginzburg/Calabi--Yau correspondence for quintic threefolds via symplectic transformations},
	journal={Invent. Math.},
	volume={182},
	date={2010},
	number={1},
	pages={117-165},
}

\bib{Chiodo2}{article}{
	author={Chiodo, Alessandro},
	author={Ruan, Yongbin},
	title={A global mirror symmetry framework for the Landau--Ginzburg/Calabi--Yau correspondence},
	journal={to appear in the Special Volume of the Ann. Inst. Fourier on the Workshop on Geometry and Physics of the Landau--Ginzburg model, available at arXiv:1307.0939},
	volume={},
	date={},
	number={},
	pages={},
}  







%

\bib{Honglu5}{article}{
	author={Fan, Honglu},
	author={Lee, YP},
	title={Towards a quantum Lefschetz hyperplane theorem in all genera},
	journal={Geometry and Topology},
	volume={23},
	date={2019},
	number={1},
	pages={493-512},
}

\bib{FJRW1}{article}{
	author={Fan, Huijun},
	author={Jarvis, Tyler},
	author={Ruan, Yongbin},
	title={The Witten equation, mirror symmetry and quantum singularity theory},
	journal={Ann. of Math.},
	volume={178},
	date={2013},
	number={1},
	pages={1-106},
}   

\bib{FJRW2}{article}{
	author={Fan, Huijun},
	author={Jarvis, Tyler},
	author={Ruan, Yongbin},
	title={The Witten equation and its virtual fundamental cycle},
	journal={available at arXiv:0712.4025},
	volume={},
	date={},
	number={},
	pages={},
} 



\bib{Gi}{article}{
	author={Givental, Alexander B.},
	title={A mirror theorem for toric complete intersections},
	journal={Topological field theory, primitive forms and related topics, Progr. Math.},
	volume={160},
	date={Kyoto, 1996},
	number={},
	pages={141-175},
}


\bib{Givental3}{article}{
	author={Givental, Alexander B.},
	title={Permutation-equivariant quantum K-theory},
	journal={Series of papers available on arXiv},
	volume={},
	date={},
	number={},
	pages={},
} 



\bib{GraberPandha}{article}{
	author={Graber, Tom},
	author={Pandaripande, Rahul},
	title={Localization of virtual classes},
	journal={Inventiones Mathematicae},
	volume={135},
	date={1999},
	number={2},
	pages={487-518},
}



\bib{Guere1}{article}{
	author={Gu\'er\'e, J\'er\'emy},
	title={A Landau--Ginzburg Mirror Theorem without Concavity},
	journal={Duke Mathematical Journal},
	volume={165},
	date={2016},
	number={13},
	pages={2461-2527},
}

%

\bib{Guere10}{article}{
	author={Gu\'er\'e, J\'er\'emy},
	title={Hodge--Gromov--Witten theory},
	journal={available at arXiv:1908.11409},
	volume={},
	date={},
	number={},
	pages={},
}

\bib{Ruan5}{article}{
	author={Guo, Shuai},
	author={Janda, Felix},
	author={Ruan, Yongbin},
	title={Structure of Higher Genus Gromov--Witten Invariants of Quintic 3-folds},
	journal={available at arXiv:1812.11908},
	volume={},
	date={},
	number={},
	pages={},
}

\bib{Machine}{article}{
	author={He, Yang-Hui},
	title={The Calabi--Yau landscape: from geometry, to physics, to machine-learning},
	journal={available at arXiv:1812.02893},
	volume={},
	date={},
	number={},
	pages={},
}




\bib{phys}{article}{
      author={Jockers, H.},
      author={Mayr, P.},
      title={Quantum K-theory of Calabi--Yau manifolds},
      journal={J. High Energ. Phys.},
      volume={11},
      date={2019},
      number={},
      pages={},
    }




\bib{Kondo}{article}{
	author={Kondo, Shigeyuki},
	title={The maximum order of finite groups of automorphisms of K3 surfaces},
	journal={Amer. J. Math.},
	volume={121},
	date={1999},
	number={6},
	pages={1245-1252},
}


\bib{KS}{article}{
	author={Kreuzer, Maximilian},
	author={Skarke, Harald},
	title={On the classification of quasihomogeneous functions},
	journal={Comm. Math. Phys.},
	volume={150},
	date={1992},
	number={1},
	pages={137-147},
}




\bib{YP}{article}{
	author={Lee, YP},
	title={Quantum K-Theory I: Foundations},
	journal={Duke Math. J.},
	volume={121},
	date={2004},
	number={3},
	pages={389-424},
}

\bib{LeeQu}{article}{
	author={Lee, YP},
	author={Qu, F},
	title={Euler characteristics of universal cotangent line bundles on $\overline{\cM}_{1,n}$},
	journal={Proc. Amer. Math. Soc.},
	volume={142},
	date={2014},
	number={2},
	pages={429-440},
}


\bib{Lho}{article}{
	author={Lho, Hyenho},
	author={Pandaripande, Rahul},
	title={Holomorphic anomaly equations for the formal quintic},
	journal={to appear in Peking Mathematical Journal},
	volume={},
	date={2019},
	number={},
	pages={},
}

\bib{LLY}{article}{
	author={Lian, Bong H.},
	author={Liu, Kefeng},
	author={Yau, Shing-Tung},
	title={Mirror principle. I},
	journal={Asian J. Math.},
	volume={1},
	date={1997},
	number={4},
	pages={729-763},
}

\bib{LiuXu}{article}{
	author={Liu, Kefeng},
	author={Xu, Hao},
	title={Intersection Numbers and Automorphisms of Stable Curves},
	journal={The Michigan Mathematical Journal},
	volume={58},
	date={2009},
	number={},
	pages={},
}


  
\bib{curve1}{article}{
	author={Okounkov, Andrei},
	author={Pandaripande, Rahul},
	title={Gromov--Witten theory, Hurwitz numbers, and completed cycles},
	journal={Ann. of Math.},
	volume={163},
	date={2006},
	number={},
	pages={517-560},
}		
		
	
\bib{curve3}{article}{
	author={Okounkov, Andrei},
	author={Pandaripande, Rahul},
	title={Virasoro constraints for target curves},
	journal={Invent. Math.},
	volume={163},
	date={2006},
	number={},
	pages={47-108},
}

\bib{Oguiso}{article}{
	author={Oguiso, Keiji},
	author={Yu, Xun},
	title={Automorphism groups of smooth quintic threefolds},
	journal={Asian Journal of Math.},
	volume={23},
	date={2019},
	number={2},
	pages={201-256},
}
   
\bib{Polish1}{article}{
	author={Polishchuk, Alexander},
	author={Vaintrob, Arkady},
	title={Matrix factorizations and cohomological field theories},
	journal={J. Reine Angew. Math.},
	volume={714},
	date={2016},
	number={},
	pages={1-122},
} 

\bib{Batyrev}{article}{
	author={Tehrani, Mahammad Farajzadeh},
	title={Automorphism group of Batyrev Calabi--Yau threefolds},
	journal={Manuscripta Mathematica},
	volume={146},
	date={2015},
	number={},
	pages={299-306},
} 


\bib{Tho1}{article}{
	author={Thomason, R},
	title={Equivariant resolution, linearization, and Hilbert's fourteenth problem over arbitrary base schemes},
	journal={Adv. in Math.},
	volume={65},
	date={1987},
	number={},
	pages={16-34},
} 

\bib{Tho2}{article}{
	author={Thomason, R},
	title={Une formule de Lefschetz en K-th\'eorie \'equivariante alg\'ebrique},
	journal={Duke Math. J.},
	volume={68},
	date={1992},
	number={},
	pages={447-462},
} 


\bib{Tonita}{article}{
	author={Tonita, Valentin},
	title={A virtual Kawasaki--Riemann--Roch formula},
	journal={Pacic J. Math.},
	volume={268},
	date={2014},
	number={1},
	pages={249-255},
} 


%













































%
%
%
%
%
%
%
%
%
%
%
%
%



\end{biblist}
\end{bibdiv}
\end{document}